\documentclass[]{interact}

\usepackage{epstopdf}
\usepackage{subfigure}

\usepackage[numbers,sort&compress]{natbib}
\bibpunct[, ]{[}{]}{,}{n}{,}{,}

\theoremstyle{plain}
\newtheorem{theorem}{Theorem}[section]
\newtheorem{lemma}[theorem]{Lemma}
\newtheorem{corollary}[theorem]{Corollary}
\newtheorem{proposition}[theorem]{Proposition}

\theoremstyle{definition}
\newtheorem{definition}[theorem]{Definition}
\newtheorem{example}[theorem]{Example}

\theoremstyle{remark}
\newtheorem{remark}{Remark}

\usepackage[ruled]{algorithm2e}
\DeclareMathOperator{\supp}{supp}
\newcommand{\matriz}[1]{\mathcal{#1}}
\newcommand{\indice}[1]{\mathbb{#1}}
\usepackage{soul,url}
\usepackage[dvipsnames]{xcolor}
\newcommand{\refre}[1]{#1}
\begin{document}

\title{Optimization Methods for Dirichlet Control Problems
}

\author{
\name{Mariano Mateos}
\affil{Dpto. de Matem\'{a}ticas,
Universidad de Oviedo, Campus de Gij\'on, 33203 Gij\'on, Spain \thanks{
Email: mmateos@uniovi.es.}}}

\maketitle

\begin{abstract}
We discuss several optimization procedures to solve finite element approximations of linear-quadratic Dirichlet optimal control problems governed by an elliptic partial differential equation posed on a 2D or 3D Lipschitz domain. The control is discretized explicitely using continuous piecewise linear approximations. Unconstrained, control-constrained, state-constrained and control-and-state constrained problems are analyzed. A preconditioned conjugate method for a reduced problem in the control variable is proposed to solve the unconstrained problem, whereas semismooth Newton methods are discussed for the solution of constrained problems. State constraints are treated via a Moreau-Yosida penalization. Convergence is studied for both the continuous problems and the finite dimensional approximations. In the finite dimensional case, we are able to show convergence of the optimization procedures even in the absence of Tikhonov regularization parameter. Computational aspects are also treated and several numerical examples are included to illustrate the theoretical results.
\end{abstract}


\begin{amscode}
 49J20, 
49M05, 
49M15,  
65N30 
\end{amscode}
\medskip

\pagestyle{myheadings}
\thispagestyle{plain}
\markboth{M. Mateos} {Optimization for Dirichlet Control Problems}

\section{Introduction}
\label{S0}
Only in the last ten years it has been possible to develop a systematic study of Dirichlet optimal control problems governed by elliptic equations, which started with the seminal paper \cite{Casas-Raymond2006}. In that work a 2D control-constrained problem governed by a semilinear equation posed on a convex polygonal domain is studied. Several other works have been published about numerical error estimates; see \cite{Deckelnick-Gunther-Hinze2009} {or \cite[Ch 2.1]{Gunther2010} } for a variational approach to control-constrained problems posed on smooth 2D or 3D domains, \cite{MayRannacherVexler2013} for a superconvergence result on the state approximation for unconstrained 2D problems, \cite{Mateos-Neitzel2015} for control-and-state constrained problems posed on convex polygonal domains. In \cite{Casas-Sokolowski2010} the authors study the control-constrained problem in a 2D smooth convex domain taking into account the problems derived by the boundary approximation and in \cite{Casas-Gunther-Mateos2011} an apparent paradox between \cite{Deckelnick-Gunther-Hinze2009} and \cite{Casas-Sokolowski2010} is explained. The regularity of the solution in possibly nonconvex polygonal plane domains is studied in \cite{AMPR2015}; see also the introduction of that paper for further references about related problems. {In the recent publication \cite{Gong-Hinze-Zhou2016}, error estimates for a Dirichlet control problem governed by a parabolic equation are obtained. In this work, the spatial discretization of the control is studied in both the cases of continuous piecewise linear finite elements and variational approach.}

Just before that, several works dealing with efficient optimization methods for control or state constrained problems had appeared; in \cite{Chen-Nashed-Qi2000, Bergounioux-Ito-Kunisch1999, Bergounioux-Kunisch2002, Kunisch-Rosch2002,  Ulbrich2003,Hintermuller-Ito-Kunisch2003, Hintermuller-Kunisch2006} the semismooth Newton method is thoroughly studied. Since all these papers about optimization are previous to the papers about the numerical analysis of Dirichlet control problems, very little or no reference is made in them to their applicability to the problems we are going to deal with. Only in \cite{Hintermuller-Kunisch2006} two different Dirichlet control problems with more regular solutions than the ones we treat here are studied in the infinite-dimensional case. Let us also mention that in \cite{Casas-Mateos-Raymond-2009} the Dirichlet boundary condition is replaced by a Robin penalization. For this kind of penalization the methods developed in the aforementioned references are directly applicable. {In \cite{Hinze-Vierling2012}, the semismooth Newton method is studied in the context of the variational approach of the control. Although the authors only exemplify their results through distributed and Robin boundary control, a combination of their results and some of the results we present in Section \ref{S4} can be applied to Dirichlet control. See Remark \ref{R41} below.}

In this work we describe optimization methods for Dirichlet control problems in both the infinite and finite dimensional cases. Convergence proofs, examples and practical implementation details are discussed for all the algorithms through the paper.

In Section \ref{S1} we state the problem and prove that it is well posed in Lipschitz domains; see Lemma \ref{L2.1}. So far, only smooth, convex, polygonal or polyhedral domains had been studied.

Next, we discretize the problem. \refre{As is usual in control problems, we have three choices to discretize the control: the first option is not to discretize the control, using a variational discretization as introduced in \cite{Hinze2005} for distributed problems; as a second option we may use piecewise constant functions; and finally we can use continuous piecewise linear functions.}

 {The choice is not trivial because first order optimality conditions for Dirichlet control problems involve the normal derivative of the adjoint state.

 If we discretize directly the optimality system using e.g. continuous piecewise linear finite elements, we would obtain two different approximations of the control: the trace of the discrete state would be continuous piecewise linear and the normal derivative of the discrete adjoint state would be piecewise constant.}
{{In \cite[Ch. 3.2.7.3]{Hinze_Pinnau_Ulbrich_Ulbrich2009} and in \cite[Example 2.1.11]{Gunther2010} this difficulty is solved using a mixed formulation of the state equation which is discretized with the lowest order Raviart-Thomas element. In \cite{Gong-Yan2011} a convergence proof for this kind of discretization is given.} \refre{In this way, both the trace of the discrete state and the normal derivative of the discrete adjoint state are piecewise constant functions, so the identification of both with the discrete control is meaningful. The discretization of the control in the presence of control constraints is carried out in these papers using a variational discretization. This kind of approximation may be convenient when the gradient of the state is the variable of interest, since this quantity is treated as one of the variables of the problem.
}}

{Another approach, in finite dimension, is to use the variational discrete normal derivative of the discrete adjoint state introduced in \cite{Casas-Raymond2006},} \refre{which is a continuous piecewise linear function. Doing so, both the trace of the discrete state and the normal derivative of the discrete adjoint state are continuous piecewise linear functions, so the identification of both with the discrete control is meaningful. Following this idea, in \cite{Deckelnick-Gunther-Hinze2009} the control is not discretized for the control-constrained case, but a variational approach is followed.}

\refre{In this work we investigate optimization methods for the case of discretizing the control explicitely using continuous piecewise linear functions.}

{The end of Section \ref{S1} is devoted to} describe with some detail some computational aspects that will be important in the rest of the work.

We next provide in Section \ref{S3} an efficient method to solve the unconstrained problem.  We propose in Subsection \ref{SS32} a preconditioned conjugate gradient (pcg in the rest of the work) method for a reduced problem in the spirit of \cite[Section 5]{Casas-Mateos-Troeltz2005}. We are able to prove convergence of the conjugate gradient even for the case where the Tikhonov regularization parameter $\nu$ vanishes.

In sections \ref{S4}, \ref{S5} and \ref{S6} we study the convergence of the semismooth Newton method for the constrained problems and write practical algorithms for the solution of the finite dimensional approximations.
In Section \ref{S4} we deal with the control-constrained problem. In many of the aforementioned references about the semismooth Newton method for control-constrained problems the authors study convergence of an abstract problem or an infinite dimensional problem. For distributed and Neumann control problems the results are immediately applicable to the finite dimensional approximations because the controls can be discretized using piecewise constant functions, and therefore the variational inequality arising from first order necessary optimality conditions can be written in an element-wise form. {The same idea applies when we are dealing with a variational approach as proposed in \refre{\cite{Deckelnick-Gunther-Hinze2009,Gunther2010,Hinze2005}} or a mixed formulation, as studied in \refre{\cite{Gong-Yan2011,Hinze_Pinnau_Ulbrich_Ulbrich2009,Gunther2010}}.} 
When we use continuous piecewise linear elements,
 the variational inequality cannot be written in a point-wise or elementwise form; see  \eqref{E418d} and Remark \ref{naive}. We include the analysis of Newton methods for both the original problem and the discretized one.

In Section \ref{S5} we study the state-constrained problem using a Moreau-Yosida penalization. Since there are no control constraints, the analysis of the semismooth Newton method for the infinite dimensional problem is applicable to the finite dimensional one, so we do not need to repeat it. We prove that the finite-element approximation of the penalized problems converge to the solution of the penalized problem. This result cannot be deduced straightforward from the ones in the literature since the penalized functional is not of class $C^2$. A continuation strategy as proposed in \cite{Hintermuller-Kunisch2006} is developed. Finally, in Section \ref{S6} we discuss the problem with both control and state constraints.

It is well known that the main difficulty with Dirichlet control problems is the low regularity of the solutions. This regularity and related error estimates are, so far, well established in 2D polygonal domains {\cite{Casas-Raymond2006,Mateos-Neitzel2015,AMPR2015}} and 2D or 3D smooth domains {\cite{Deckelnick-Gunther-Hinze2009}} but there is not, up to our best knowledge, a general study in 3D polyhedral domains. Although the main focus of this work is on optimization methods, we also study the regularity of the solution and error estimates of the approximations in one example case in a polyhedron; see Example \ref{Ex3D}.

\section{Statement of the problem and first properties}
\label{S1}
Let $\Omega\subset\mathbb{R}^d$, $d=2$ or $d=3$, be a bounded domain with  Lipschitz boundary $\Gamma$  and in this domain consider a target state $y_\Omega\in L^2(\Omega)$. Consider also the continuous linear operator $S:L^2(\Gamma)\to L^2(\Omega)$ such that $y=Su$ if and only if $y$ is the solution in the transposition sense (see Definition \ref{D1} below)  of
\begin{equation}\label{E1}-\Delta y = 0\mbox{ in }\Omega,\ y=u\mbox{ on }\Gamma.\end{equation}
Let $\omega\subset\Omega$ be a domain such that $\bar\omega\subset\Omega$ and define two pairs of functions $\alpha,\ \beta\in C(\Gamma)$ and $a,\ b\in C(\bar\omega)$ such that $\alpha(x)<\beta(x)$ for all $x\in\Gamma$ and $a(x)<b(x)$ for all $x\in\bar\omega$. For some fixed regularization parameter $\nu\geq0$, define
\[J(u) =\frac12 \|Su-y_\Omega\|^2_{L^2(\Omega)}+\frac{\nu}{2}\|u\|^2_{L^2(\Gamma)}\]
and consider the sets
\[U_{\alpha,\beta}=\{u\in L^2(\Gamma):\ \alpha(x)\leq u(x)\leq \beta(x)\mbox{ for a.e. }x\in\Gamma\},\]
and
\[K_{a,b}=\{y\in L^{2}(\Omega)\cap C(\bar\omega):\ a(x)\leq y(x)\leq \beta(x)\mbox{ for a.e. }x\in \bar\omega\}.\]

In this work we will study optimization procedures for
the following four Dirichlet control problems:
\[
(P^U)\min_{u\in L^2(\Gamma)} J(u),\ \ \ \
(P^C)\min_{u\in U_{\alpha,\beta}}J(u),\ \ \ \
(P^S)\min_{Su\in K_{a,b}}J(u),\ \ \ \
(P^{CS})\hspace{-3mm}\min_{\begin{array}{c}
\scriptstyle{u\in U_{\alpha,\beta}}\\
\scriptstyle{ Su\in K_{a,b}}\end{array}}J(u),\]
namely the unconstrained, control-constrained, state-constrained and control-and-state constrained problems.
\begin{remark}Almost all the literature related to these problems is written using the state equation \eqref{E1}. There would be no problem in taking into account an equation of the form
\[Ay :=-\sum_{i=1}^d\partial_{i}(a_{i,j}\partial_j y) + a_0 y = F\mbox{ in }\Omega,  \ y = u + G\mbox{ on }\Gamma\]
with regular enough $F$, $G$,  $a_0\geq 0$ and $a_{i,j}=a_{j,i}$  satysfing an uniform ellipticity condition.
\end{remark}

Let us state precisely what we mean by solution in the transposition sense. Consider the space
\[\Phi = \{\phi: \phi\in H^1_0(\Omega)\mbox{ and }\Delta\phi \in L^2(\Omega)\}.\]
This is a Banach space with the graph norm $\|\phi\|_\Phi = \|\phi\|_{H^1(\Omega)}+ \|\Delta\phi\|_{L^2(\Omega)}$.
Further, the functions in this space satisfy $\partial_n \phi\in L^2(\Gamma)$. This is known to be true for smooth domains; convex domains, see \cite[Lemma A.1]{Casas-Mateos-Raymond-2009}; plane polygonal domains, see \cite[Corollary 2.3]{AMPR2015}; or polyhedral domains, see \cite[Theorem 2.6.7]{Grisvard1992} and the usual trace theorem. We have not been able to find a proof of this fact for general Lipschitz domains.  In \cite[Theorem B.2]{Jerison-Kenig1995} the regularity $\phi\in H^{3/2}(\Omega)$ is proved. Nevertheless, as the authors notice in page 165 of this reference, the trace theorem is not valid neither in $H^{3/2}(\Omega)$ nor in $H^{1/2}(\Omega)$, so we cannot deduce immediately that $\partial_n\phi \in L^2(\Gamma)$. The results in \cite{Jerison-Kenig1981N,Jerison-Kenig1995} imply that the usual trace result can be extended to harmonic functions in the limit cases. We show  next how to take advantage of this  to prove that $\partial_n\phi\in L^2(\Gamma)$ in Lipschitz domains. Regarding the analysis of semismooth Newton methods --see Lemma \ref{L4.1} below-- we also prove $L^q(\Gamma)$ regularity for some $q>2$.
\begin{lemma}\label{L2.1}Let $\Omega\subset\mathbb{R}^d$, $d=2$ or $d=3$, be a bounded domain with  Lipschitz boundary $\Gamma$ and consider $\phi\in \Phi$. Then, there exists $q_0>2$ depending on the domain such that, for $2\leq q<q_0$, we have
$\partial_n \phi\in L^q(\Gamma)$ and
\begin{equation}\label{E2.1}\|\partial_n \phi\|_{L^q(\Gamma)}\leq C \|\Delta\phi\|_{L^2(\Omega)}.\end{equation}

If, further, $\Omega$ is smooth or convex or polygonal or polyhedral, then there exists $t>0$ such that $\partial_n \phi\in H^t(\Gamma)$ and
\begin{equation}\label{E2.1a}\|\partial_n \phi\|_{H^t(\Gamma)}\leq C \|\Delta\phi\|_{L^2(\Omega)}.\end{equation}
\end{lemma}

\begin{proof}
Denote $z=-\Delta\phi$, extend $z$ by zero to $\mathbb{R}^d$, consider the Newtonian potential centered at the origin $N(x)$ and define $w = z * N$, the convolution product of $z$ and $N$. Then $w\in H^{2}(\Omega)$ and $\nabla w\in H^{1}(\Omega)^d$, so it is clear that $\mathrm{Tr} (\nabla w)\in H^{1/2}(\Gamma)^d\hookrightarrow L^{q}(\Gamma)^d$ for all $q<\infty$ if $d=2$ and all $q\leq 4 $ if $d=3$, since the dimension of $\Gamma$ is $d-1$. This implies that:

(a) $\partial_n w = \nabla w\cdot n\in L^{q}(\Gamma)$ because the boundary $\Gamma$ is Lipschitz and therefore it has a unit normal vector defined almost everywhere; and there exists $C>0$ such that
\begin{equation}\label{Ezz1}\|\partial_n w\|_{L^q(\Gamma)}\leq C\|\Delta\phi\|_{L^2(\Omega)};\end{equation}

(b) $g = \mathrm{Tr}(w)\in W^{1,q}(\Gamma)$ due precisely to the definition of $W^{1,q}(\Gamma)$, and there exists $C>0$ such that
\begin{equation}\label{Ezz2}
\|g\|_{W^{1,q}(\Gamma)}\leq C \|\Delta\phi\|_{L^2(\Omega)}.
\end{equation}
Define now $v\in H^1(\Omega)$ the unique solution of $-\Delta v = 0$ in $\Omega$, $v = g$ on $\Gamma$.
Using \cite[Theorem 5.6]{Jerison-Kenig1995}, we have that there exists $q_0>2$ such that if $2\leq q <q_0$, then the nontangential maximal function of the gradient of $v$ satisfies $\mathrm{M}(\nabla v)\in L^q(\Gamma)$ and there exists $C>0$ such that
\begin{equation}\label{Ezz3}
\|\mathrm{M}(\nabla v)\|_{L^q(\Gamma)}\leq C\|g\|_{W^{1,q}(\Gamma)}.
\end{equation}
As is pointed out in \cite[p. 438]{Dahlberg-Kenig1987}, this implies that $\nabla v$ has nontangential limit a.e. on $\Gamma$ and we can define the normal derivative of $v$ at a point $s\in\Gamma$  as the nontangential limit as $x\to s$ of $\nabla v(x)\cdot n(s)$. For a precise definition of nontangential limit and nontangential maximal function see, e.g., the introduction of the work \cite{Dahlberg-Kenig1987}. This, together with inequalities \eqref{Ezz3} and \eqref{Ezz2} imply that $\partial_n v \in L^q(\Omega)$ and
\begin{equation}\label{Ezz4}
\|\partial_n v\|_{L^{q}(\Gamma)}\leq C \|\Delta\phi\|_{L^2(\Omega)}.
\end{equation}

So we have that $\phi = v-w$ and we can define in a natural way $\partial_n \phi  = \partial_n v -\partial_n w\in L^q(\Gamma)$. The estimate \eqref{E2.1} follows from \eqref{Ezz1} and \eqref{Ezz4}.

For smooth, convex, polygonal or polyhedral domains, the second result follows from the regularity $\phi\in H^{3/2+t}(\Omega)$ and the usual trace theorem for $\nabla \phi$. See \cite[Lemma A.1]{Casas-Mateos-Raymond-2009} for convex domains,  \cite[Corollary 2.3]{AMPR2015} for plane polygonal domains and \cite[Theorem 2.6.7]{Grisvard1992} for polyhedral domains.
\end{proof}

\begin{definition}\label{D1}We will say that $y\in L^2(\Omega)$ is the solution in the transposition sense of \eqref{E1} if
\begin{equation}\label{E2}(y,-\Delta\phi)_\Omega = -(u,\partial_n \phi)_\Gamma\mbox{ for all }\phi\in\Phi.\end{equation}
\end{definition}
Here and in the rest of the work $(\cdot,\cdot)_X$  stands for the standard inner product in $L^2(X)$.

The adjoint operator of $S$ is $S^*:L^{2}(\Omega)\to L^2(\Gamma)$ defined by $S^*z = -\partial_n\phi$, where $\phi$ is the unique weak solution of
\begin{equation}\label{E3}-\Delta \phi = z\mbox{ in }\Omega,\ \phi = 0\mbox{ on }\Gamma.\end{equation}
We can write now
\begin{eqnarray*}
J(u) &=&\frac12(Su-y_\Omega,Su-y_\Omega)_\Omega+\frac{\nu}2(u,u)_\Gamma =\frac12(S^*Su+\nu u,u)_\Gamma - (S^*y_\Omega,u)_\Gamma + c_\Omega
\end{eqnarray*}
where $c_\Omega=0.5 \|y_\Omega\|_{L^2(\Omega)}^2$ is a constant.

From this expression, we can easily compute the derivative of $J$ at a point $u\in L^2(\Gamma)$ in the direction $v\in L^2(\Gamma)$:
\[J'(u)v=(S^*Su+\nu u,v)_\Gamma -(S^*y_\Omega,v)_\Gamma.\]
For later use, we will define now for every $u\in L^2(\Gamma)$, $y_u=Su\in H^{1/2}(\Gamma)$ and $\varphi_u\in H^1_0( \Omega)$ the unique solution of
\[-\Delta \varphi_u = y_u-y_\Omega\mbox{ in }\Omega,\ \varphi_u=0\mbox{ on }\Gamma.\]

\subsection*{Discretization.}

To discretize the problems. consider  $\{\mathcal{T}_h\}_h$ a regular family of triangulations of $\bar\Omega$. To simplify the notation, we will suppose that $\Gamma$ is polygonal or polyhedral. Related to the mesh, we will call $N$ the number of nodes and $\{x_j\}_{j=1}^N$ the nodes of the mesh. We define the sets of interior indexes, boundary indexes and indexes in $\bar\omega$ as $\indice{I} = \{j:x_j\in\Omega\}$, $\indice{B}=\{j:x_j\in\Gamma\}$ and $\indice{J} = \{j:x_j\in\bar\omega\}$.
For the discretization of the state and the adjoint state we use the space of linear finite elements $Y_{h}\subset H^1(\Omega)$,
\[
Y_{h}=\{y_h\in C(\bar\Omega)\colon y_h\in P^1(T)\ \forall T\in \mathcal{T}_h\}.
\]
As usual, we will abbreviate $Y_{h0}=Y_h\cap H^1_0(\Omega)$. For the control we use the space $U_h$ of continuous piecewise linear functions on $\Gamma$
\[U_h = \{u_h\in C(\Gamma)\colon u_h\in P^1(T\cap\Gamma)\ \forall T\in\mathcal{T}_h\}.
\]
Notice that the elements of $U_h$ are the traces of the
elements of $Y_h$.
We will denote $I_h:C(\bar\Omega)\to Y_h$ or $I_h:C(\Gamma)\to U_h$ the interpolation operator related to these spaces and $\Pi_h:L^2(\Omega)\to Y_h$ or $\Pi_h:L^2(\Gamma)\to U_h$ the projection onto this spaces in the $L^2$ sense. For all $y\in L^2(\Omega)$ and all $u\in L^2(\Gamma)$:
\[(\Pi_h y,y_h)_\Omega = (y,y_h)_\Omega\ \forall y_h\in Y_h\mbox{ and }
(\Pi_h u,u_h)_\Gamma = (u,u_h)_\Gamma\ \forall u_h\in U_h.\]

We discretize the state equation following the work by Berggren \cite{Berggren2004}: define $S_h:L^2(\Gamma)\to L^2(\Omega)$ such that for $u\in L^2(\Gamma)$, $y_h=S_hu$ if and only if $y_h\in Y_h$ is  the unique solution of
\begin{equation}\label{E6}a(y_h, z_h)=0\ \forall z_h\in Y_{h0},\ (y_h,v_h)_{\Gamma} = (u,v_h)_\Gamma\ \forall v_h\in U_h,\end{equation}
where $a(\cdot,\cdot)$ is the bilinear form associated to the operator in the PDE. In the case of the Laplace operator $a(y,z)=\int_\Omega\nabla^T y\nabla z dx$.
The discrete functional is thus defined as
\[J_h(u)=\frac12(S_hu-y_\Omega,S_hu-y_\Omega)_\Omega+\frac{\nu}2(u,u)_\Gamma.\]

We define now
\[U^h_{\alpha,\beta} = \{u_h\in U_h:\ \alpha(x_j)\leq u_h(x_j)\leq \beta(x_j)\ \forall j\in\indice{B}\},\]
 and
\[K^h_{a,b} = \{y_h\in Y_h:\ a(x_j)\leq y_h(x_j)\leq b(x_j)\ \forall j\in \indice{J}\}.\]
The discrete problems reads thus as
\[
(P_h^U)\min_{u_h\in U_h} J_h(u_h),\ \ \ \
(P_h^C)\min_{u\in U^h_{\alpha,\beta}}J_h(u_h),\ \ \ \
(P_h^S)\min_{S_hu_h\in K^h_{a,b}}J_h(u_h),\ \ \ \
(P_h^{CS})\hspace{-3mm}\min_{\begin{array}{c}
\scriptstyle{u_h\in U^h_{\alpha,\beta}}\\
\scriptstyle{ S_hu_h\in K^h_{a,b}}\end{array}}J_h(u_h).\]

The adjoint operator of $S_h$ is given by the discrete variational normal derivative. See \cite{Casas-Raymond2006}. $S_h^*:L^2(\Omega)\to L^2(\Gamma)$ and $w_h=S_h^*y$ if $w_h\in U_h$ satisfies
\begin{equation}(w_h,z_h)_\Gamma =  (y,z_h)_\Omega - a(z_h,\phi_h)\mbox{ for all }z_h\in Y_h,\label{E50}\end{equation}
where $\phi_h\in Y_{h0}$ is the unique solution of
\begin{equation}a(z_h, \phi_h) = (y,z_h)_\Omega\mbox{ for all } z_h\in Y_{h0}.\label{E100}\end{equation}
It is customary to write $S_h^*y=-\partial_n^h\phi_h$. For $u\in L^2(\Gamma)$ and $y\in L^2(\Omega)$, we have
\[(S_hu,y)_\Omega = (u,S^*_h y)_\Gamma.\]
We can then write
\[J_h(u) =\frac12(S^*_hS_hu+\nu u,u)_\Gamma -(S_h^* y_\Omega,u)_\Gamma+c_\Omega.\]
The computation of the derivative of $J_h$ at a point $u\in L^2(\Gamma)$ in the direction $v\in L^2(\Gamma)$ is then obtained with
\[J'_h(u)v =(S_h^*S_hu+\nu u,v)_\Gamma - (S^*_hy_\Omega,v)_\Gamma.\]
Since our final objective is to optimize in $U_h$, let us see in more detail how to make the computations in this space.

Consider $\{e_j\}_1^{N}$ the nodal basis in $Y_h$, where $N$ is the dimension of $Y_h$ and satisfies $N=N_\indice{I}+N_\indice{B}$, the latter being respectively the number of interior and boundary nodes. With an abuse of notation, we will also denote $e_j$ the restriction of $e_j$ to $\Gamma$. Define the usual finite element stress, mass and boundary mass matrices by
\[\matriz{K}_{i,j}=a( e_i,e_j),\  \matriz{M}_{i,j}=(e_i,e_j)_\Omega,\ \matriz{B}_{i,j} =(e_i,e_j)_\Gamma \mbox{ for }1\leq i,j\leq N.\]
We will also use $\matriz{I}\in\mathbb{R}^{N\times N}$ for the identity matrix, $\matriz{O}\in\mathbb{R}^{N\times N}$ for the zero matrix and usually refer to submatrices as $\matriz{K}_{\indice{I},\indice{I}}$ or $\matriz{K}_{\indice{I},:}$ defined by taking the rows or columns designed by the sets of indexes in the subscripts, the semicolon meaning ``all the indexes''. For instace, the usual boundary mass matrix is $\matriz{B}_{\indice{B},\indice{B}}$.

Given $u_h=\sum_{j\in\indice{B}}u_j e_j$, we denote $\bm{u} \in\mathbb{R}^{N_\indice{B}\times 1}$ the vector $(u_1,\dots,u_{N_\indice{B}})^T$ and for $y_h=\sum_{j=1}^N y_je_j$ we denote $\bm{y} \in\mathbb{R}^{N\times 1}$ the vector $(y_1,\dots,y_N)^T$. Using \eqref{E6}, we have that $y_h=S_hu_h$ iff
\begin{equation}\label{E250}
\left[\begin{array}{cc}\matriz{K}_{\indice{I},\indice{I}} & \matriz{K}_{\indice{I},\indice{B}}\\
\matriz{O}_{\indice{B},\indice{I}} & \matriz{B}_{\indice{B},\indice{B}}\end{array}\right]
\left[
\begin{array}{c}\bm{y}_{\indice{I}}\\ \bm{y}_{\indice{B}}
\end{array}
\right]
=
\left[
\begin{array}{c} \bm{0}\\ \matriz{B}_{\indice{B},\indice{B}} \bm{u}
\end{array}\right].
\end{equation}
Since $\matriz{B}_{\indice{B},\indice{B}}$ is nonsingular, we can write this as
\begin{equation}\label{E200}
\begin{array}{rcr}\matriz{K}_{\indice{I},\indice{I}} \bm{y}_{\indice{I}}&=&
-\matriz{K}_{\indice{I},\indice{B}} \bm{u},\\
 \bm{y}_{\indice{B}}&=&
\bm{u}.
\end{array}
\end{equation}
If we define $\matriz{S}\in\mathbb{R}^{N_{\indice{B}}\times N}$ as
\[\matriz{S}=\left[\begin{array}{cc}\matriz{K}_{\indice{I},\indice{I}} & \matriz{K}_{\indice{I},\indice{B}}\\
\matriz{O}_{\indice{B},\indice{I}} & \matriz{I}_{\indice{B},\indice{B}}\end{array}\right]^{-1}
\matriz{I}_{:,\indice{B}}
\]
we have that $y_h=S_hu_h$ if and only if $\bm{y}=\matriz{S}\bm{u}$.

Given $y_h\in Y_h$, let $\phi_h=\sum_{j\in \indice{I}}\phi_j e_j$ be the solution of \eqref{E100} for $y=y_h$. Denoting $\bm{\phi}\in\mathbb{R}^{N_\indice{I}\times 1}$ the corresponding vector, it can be computed as the solution of
\begin{equation}\label{E300}
\matriz{K}_{\indice{I},\indice{I}}\bm{\phi} = \matriz{M}_{\indice{I},:}\bm{y}.
\end{equation}
Then we could compute $\bm{w}\in \mathbb{R}^{N_\indice{B}\times 1}$, the vector whose components are the coefficients of the (minus) discrete normal derivative  $-\partial_n^h\phi_h = S_h^*y_h=w_h=\sum_{j\in\indice{B}}w_je_j$ solving the system
\begin{equation}\label{E400}
\matriz{B}_{\indice{B},\indice{B}}\bm{w} = \matriz{M}_{\indice{B},:}\bm{y} -
\matriz{K}_{\indice{B},\indice{I}}\bm\phi.
\end{equation}
Formally, we can also write that
$\bm{w} = \matriz{B}_{\indice{B},\indice{B}}^{-1} \matriz{S}^T\matriz{M}\bm{y}.$
To finish this section we also define the matrix $\matriz{A}\in \mathbb{R}^{N_{\indice{B}},\times  N_{\indice{B}}}$ and the vector $\bm{f}\in \mathbb{R}^{N_{\indice{B}}\times 1}$ by
\[\matriz{A}_{i,j}= (S_h^*S_h e_i+\nu e_i,e_j)_\Gamma\mbox{ and } f_i= (S^*_hy_\Omega,e_i)_\Gamma.\]
We have that
\begin{equation}\label{E9}J_h(u_h)=\frac12 \bm{u}^T\matriz{A}\bm{u} - \bm{f}^T\bm{u}+c_\Omega.\end{equation}
We also note that
\begin{equation}\label{E2end}\matriz{A} = \matriz{S}^T\matriz{M}\matriz{S}+\nu\matriz{B}_{\indice{B}\indice{B}}.\end{equation}
To compute the vector $\bm{f}$,  we consider the projection of $y_\Omega$ onto $Y_h$ in the $L^2(\Omega)$ sense, $y_{\Omega,h} = \Pi_h y_\Omega$ and denote $\bm{y}_\Omega\in \mathbb{R}^{N\times 1}$, the vector whose $j$-th component is $y_{\Omega,h}(x_j)$, and \begin{equation}\label{E2.9}\bm{f}=\matriz{S}^T\matriz{M}\bm{y}_\Omega.\end{equation}

So for $u_h,v_h\in U_h$, the latter represented by the vector $\bm{v}$, we can compute
\[J'(u_h)v_h=\bm{v}^T\matriz{B}_{\indice{B},\indice{B}} \bm{w} +\nu \bm{v}^T \matriz{B}_{\indice{B},\indice{B}}\bm{u}-\bm{v}^T\bm{f}.\]
Notice that applying the equality \eqref{E400}, the explicit computation of $\bm{w}$ is not necessary, and we can write
\[J'(u_h)v_h=\bm{v}^T(\matriz{M}_{\indice{B},:}\bm{y}-
\matriz{K}_{\indice{B},\indice{I}}\bm\phi +\nu  \matriz{B}_{\indice{B},\indice{B}}\bm{u}-\bm{f}). \]

\section{Unconstrained problem}
\label{S3}
Problem $(P^U)$ has a unique solution $\bar u\in L^2(\Gamma)$ that satisfies $J'(\bar u)=0$.
For every $0<h<h_0$, problem $(P^U_h)$ has also a unique solution $\bar u_h$ that satisfies $J'_h(\bar u_h)=0$. The problems being convex, these conditions are also sufficient. Moreover it is known that $\bar u_h\to\bar u$ strongly in $L^2(\Gamma)$ and also error estimates are available in terms of the mesh size $h$ if $\nu>0$ and the domain is either {2D and polygonal or smooth}. See \cite{Casas-Raymond2006,Deckelnick-Gunther-Hinze2009,Casas-Sokolowski2010,MayRannacherVexler2013,Mateos-Neitzel2015,AMPR2017}.

Let us comment two different approaches for the solution of the FEM approximation of $(P^U)$.

\subsection{Solve the optimality system for the state, the adjoint state and the control.}
We write first order optimality conditions for $(P^U_h)$. There exists some $h_0>0$ such that for every $0<h<h_0$, there exist unique $\bar u_h\in U_h$, $\bar y_h\in Y_h$ and $\bar\varphi_h\in Y_{h0}$ satisfying the optimality system:
\begin{subequations}
\begin{eqnarray}
a(\bar y_h,z_h)&=&0 \mbox{ for all }z_h\in Y_{h0},\\ \bar y_h&\equiv& \bar u_h\mbox{ on }\Gamma,\\
a( z_h, \bar\varphi_h) & =&(\bar y_h-y_\Omega,z_h)_\Omega \mbox{ for all }z_h\in Y_{h0},\\
\nu\bar u_h & \equiv&\partial_n^h\bar\varphi_h\mbox{ on }\Gamma.
\end{eqnarray}
\end{subequations}
Taking into account the definition of discrete variational normal derivative and relations \eqref{E200}--\eqref{E400}, we can write this optimality system as
\[\left[\begin{array}{cccc}
\matriz{K}_{\indice{I},\indice{I}} &\matriz{O}_{\indice{I},\indice{B}} & \matriz{K}_{\indice{I},\indice{B}} & \matriz{O}_{\indice{I},\indice{I}} \\
\matriz{O}_{\indice{B},\indice{I}} & \matriz{I}_{\indice{B},\indice{B}} & -\matriz{I}_{\indice{B},\indice{B}} & \matriz{O}_{\indice{B},\indice{I}}\\
-\matriz{M}_{\indice{I},\indice{I}} & -\matriz{M}_{\indice{I},\indice{B}} & \matriz{O}_{\indice{I},\indice{B}} & \matriz{K}_{\indice{I},\indice{I}}\\
\matriz{M}_{\indice{B},\indice{I}} & \matriz{M}_{\indice{B},\indice{B}} & {\nu}\matriz{B}_{\indice{B},\indice{B}}&
-\matriz{K}_{\indice{B},\indice{I}}
\end{array}\right]
\left[\begin{array}{c}
\bm{y}_\indice{I} \\\bm{y}_\indice{B} \\\bm{u} \\ \bm{\varphi}_\indice{I}
\end{array}
\right] =
\left[\begin{array}{c}
\bm{0}_\indice{I} \\
\bm{0}_\indice{B} \\
-\matriz{M}_{\indice{I},:}\bm{y}_\Omega\\
\matriz{M}_{\indice{B},:}\bm{y}_\Omega
\end{array}
\right].
\]
We may eliminate $\bm{u}$ with the boundary condition, and reorder the equations  to get the linear system
\begin{equation}\left[\begin{array}{cc}
\matriz{M}+\nu\matriz{B} & -\matriz{K}_{:,\indice{I}}  \\
-\matriz{K}_{\indice{I},:} & \matriz{O}_{\indice{I},\indice{I}}
\end{array}\right]
\left[\begin{array}{c}
\bm{y} \\
\bm{\varphi}_\indice{I} \end{array}
\right] =
\left[\begin{array}{c}
\matriz{M}\bm{y}_\Omega\\
\bm{0}
\end{array}
\right].
\label{E500}
\end{equation}
Notice that system \eqref{E500} is solvable even for $\nu = 0$.
Solving this equation would completely solve the problem for the unconstrained case.
When the discretization is very fine, the number of unknowns may make the solution of the system by direct methods too difficult.
The preconditioned conjugate gradient method for this kind of linear systems is studied by Sch\"oberl and Zulehner in \cite{Schoberl-Zulehner2007} and  Herzog and Sachs in \cite{Herzog-Sachs2010}. A preconditioner can be built using matrices representing scalars products in $Y_h$ and $Y_{h0}$. Following Algorithm 1 in the last-mentioned reference, at each iteration three linear systems must be solved: two of size $N$ and one of size $N_{\indice{I}}$. In \cite[Algorithm 3]{Herzog-Sachs2010}, the systems are solved using a multigrid method.

Reduced problems in the adjoint state variable and related pcg have also been studied in \cite{Schiela-Ulbrich2014}. Nevertheless, the structure of the matrix we have in \eqref{E500} is different to the one treated in that reference and application of their results is not straightforward.

\subsection{Use an iterative method to solve a reduced problem in the control variable.}\label{SS32}
Let us see another way of solving $(P^U_h)$. Using \eqref{E9}, first order optimality conditions can also be written as
\begin{equation}\matriz{A}\bm{u} = \bm{f}.\label{E600}\end{equation}
\begin{lemma}The matrix $\matriz{A}$ is symmetric and positive definite for all $\nu\geq 0$ and there exists $C>0$ independent of $h$ and $\nu$ such that its condition number is bounded by
\begin{equation}\label{E15a}\kappa( \matriz{A} ) \leq \frac{
C {
\frac{\lambda_{N_\indice{B}}(\matriz{B}_{\indice{B},\indice{B}})}
{\lambda_1(\matriz{M})}
}
\lambda_N(\matriz{M}) +\nu \lambda_{N_\indice{B}}(\matriz{B}_{\indice{B},\indice{B}})}
{\lambda_1(\matriz{M}) +\nu \lambda_{1}(\matriz{B}_{\indice{B},\indice{B}})},
\end{equation}
where $0<\lambda_1(\matriz{B}_{\indice{B},\indice{B}})< \lambda_{N_\indice{B}}(\matriz{B}_{\indice{B},\indice{B}})$ and $0<\lambda_1(\matriz{M})<\lambda_N(\matriz{M})$ are respectively the smallest and greatest eigenvalues of the matrices $\matriz{B}_{\indice{B},\indice{B}}$ and $\matriz{M}$.
\end{lemma}
\begin{proof}It is clear from \eqref{E2end} that $\matriz{A}$ is symmetric. The mass matrices $\matriz{M}$ and $\matriz{B}_{\indice{B},\indice{B}}$ are symmetric and positive definite and therefore $\lambda_{1}(\matriz{M})>0$ and $\lambda_{1}(\matriz{B}_{\indice{B},\indice{B}})>0$.  Since the boundary components of $S\bm{u}$ are exactly $\bm{u}$, we have that $\|\matriz{S}\bm{u}\|_{\mathbb{R}^N} \geq \|\bm{u}\|_{\mathbb{R}^{N_\indice{B}}}$ and hence
\begin{equation}\label{E35}
\bm{u}^T\matriz{A}\bm{u} = \bm{u}^T(\matriz{S}^T\matriz{M}\matriz{S}+\nu \matriz{B}_{\indice{B},\indice{B}}) \bm{u} =
(\matriz{S}\bm{u})^T\matriz{M}(\matriz{S}\bm{u}) +\nu\bm{u}^T \matriz{B}_{\indice{B},\indice{B}}\bm{u})\geq (\lambda_1(\matriz{M})+ \nu \lambda_1(\matriz{B}_{\indice{B},\indice{B}})) \|\bm{u}\|_{\mathbb{R}^{N_\indice{B}}}^2
\end{equation}
so $\matriz{A}$ is positive definite.

From \cite[Equation (3.8)]{MayRannacherVexler2013}, we know that there exists some $C>0$ such that $\|S_h u_h\|_{L^2(\Omega)}\leq \sqrt{C} \|u_h\|_{L^2(\Gamma)}$. Since
\[
\|S_hu_h\|_{L^2(\Omega)}^2 = (\matriz{S}\bm{u})^T\matriz{M}(\matriz{S}\bm{u})\geq \lambda_1(\matriz{M})\|\matriz{S}\bm{u}\|_{\mathbb{R}^N}^2\]
and
\[\|u_h\|_{L^2(\Gamma)}^2 = \bm{u} \matriz{B}_{\indice{B},\indice{B}}\bm{u}\leq \lambda_{N_\indice{B}}(\matriz{B}_{\indice{B},\indice{B}})\|\bm{u}\|_{\mathbb{R}^{N_\indice{B}}}^2,\]
we obtain
\[
\bm{u}^T\matriz{A}\bm{u} =
(\matriz{S}\bm{u})^T\matriz{M}(\matriz{S}\bm{u}) +\nu\bm{u}^T \matriz{B}_{\indice{B},\indice{B}}\bm{u})
\leq (C {
\frac{\lambda_{N_\indice{B}}(\matriz{B}_{\indice{B},\indice{B}})}
{\lambda_1(\matriz{M})}
}
\lambda_N(\matriz{M})+ \nu \lambda_{N_\indice{B}}(\matriz{B}_{\indice{B},\indice{B}})) \|\bm{u}\|_{\mathbb{R}^{N_\indice{B}}}^2.
\]
The bound for the condition number follows directly from the last inequality and \eqref{E35}.
\end{proof}
For a graded family of mesh with grading parameter $0<\mu\leq 1$ (see \cite{ASW1996}), we usually define $h=N^{-1/d}$, $N$ being the number of nodes of the mesh and $d$ the dimension of $\Omega$. The case $\mu=1$ corresponds to that of a quasi-uniform mesh family.
\begin{corollary}\label{C3.2}If $\mathcal{T}_h$ is a family of graded meshes with grading parameter $0<\mu\leq 1$, then there exists $C>0$ independent of $h$ and $\nu$ such that
\begin{equation}\label{E17}
\kappa(\matriz{A})\leq C h^{(2d-1)\left(1-\frac{1}{\mu}\right)}
\frac{1 + \nu h^{\left(1-\frac{1}{\mu}\right)d}}{h^{\frac{1}{\mu}}+\nu}.
\end{equation}
In particular, for a quasi-uniform family, there exists $C>0$ such that
\begin{equation}\label{E15}
\kappa(\matriz{A})\leq C\frac{1+\nu}{\nu+h}.\end{equation}
\end{corollary}
\begin{proof}
 from \cite{Kamenski-Huang-Xu2014} and \cite{ASW1996}, we have that there exists constants $0<C_1<C_2 $
\[\lambda_{N_\indice{B}}(\matriz{B}_{\indice{B},\indice{B}})\leq C_2 h^{d-1},\
\lambda_{1}(\matriz{B}_{\indice{B},\indice{B}})\geq C_1 h^{\frac{d-1}{\mu}},\
\lambda_N(\matriz{M})\leq C_2 h^d,\
 \lambda_1(\matriz{M}) \geq C_1 h^{\frac{d}{\mu}}\]
Estimate \eqref{E17} follows then from \eqref{E15a}. Estimate \eqref{E15} follows from \eqref{E17} for $\mu=1$.
\end{proof}

The explicit computation of the matrix $\matriz{A}$ is out of the question, since it requires the solution of $2N_{\indice{B}}$ systems of size $N_{\indice{I}}$. Much better performance can be achieved using an iterative method which only requires the computation of $\bm{d}=\matriz{A}\bm{u}$. This is shown in Algorithm \ref{A100}.

\LinesNumbered
\begin{algorithm}[H]
\caption{Function $\bm{d}=\matriz{A}\bm{u}$\label{A100}}
\DontPrintSemicolon
solve \eqref{E200} for $\bm{y}$\;
solve \eqref{E300} for $\bm{\phi}$\;
set $\bm{d} = \matriz{M}_{\indice{B},:}\bm{y}-
\matriz{K}_{\indice{B},\indice{I}}\bm\phi+\nu\matriz{B}_{\indice{B},\indice{B}}\bm{u}$\;
\end{algorithm}

Preconditioned conjugate gradient methods can be carried out at the cost of one evaluation of $\matriz{A}\bm{u}$ per iteration. The computation of $\bm{f}$ can be done in a similar way, but of course it only must be computed once; see Algorithm \ref{A200}.

\begin{algorithm}[H]
\caption{Computation of $\bm{f}$\label{A200}}
\DontPrintSemicolon
define $\bm{y}_\Omega$ by $y_j=y_{\Omega,h}(x_j)$ for $j=1:N$\;
set $\bm{y}=\bm{y}_\Omega$ in \eqref{E300} and solve for $\bm{\phi}$\;
set $\bm{f} = \matriz{M}_{\indice{B},:}\bm{y}_\Omega-
\matriz{K}_{\indice{B},\indice{I}}\bm\phi$\;
\end{algorithm}

To finish this section, let us say a word about a preconditioner to solve \eqref{E600}. In practice, matrices like $\matriz{B}_{\indice{B},\indice{B}}$ or  $\matriz{P}=\matriz{M}_{\indice{B},\indice{B}}+\nu\matriz{B}_{\indice{B},\indice{B}}$ make acceptable preconditioners. This is specially important when using graded meshes; see Example \ref{Ex33} below.
Notice that at each pcg iterate, we only need to solve two linear systems, each of size $N_{\indice{I}}$, to compute $\matriz{A}\bm{u}$, plus another one of size $N_{\indice{B}}$  to compute $\matriz{P}^{-1}\bm{r}$.

\begin{example}\label{Ex31}
All the code for all the examples has been made by the author using {\sc Matlab} R2015b and has been run on a desktop PC with Windows 10 Pro 64bits with 16GB RAM DIMM 1333Mhz on an Intel Core i7 CP 870@2.93Ghz.

In this example we compare the different execution times, $t_D$ and $t_P$, that we  obtain when we solve the optimality system using a direct solver for equation \eqref{E500} or using an iterative method, the preconditioned conjugate gradient method in our case, to solve the reduced problem \eqref{E600}.
We have used respectively \textsc{Matlab} builtin commands \texttt{mldivide} and \texttt{pcg}.

We use the example in \cite{Mateos-Neitzel2015}, where $\Omega$ the interior of the convex hull of the points $(-0.5,-0.5)$, $(0.5,-0.5)$, $(0.5,0)$, $(0,0.5)$, $(-0.5,0.5)$; $y_\Omega\equiv 1$ and $\nu=1$.

A rough initial mesh is built with {\sc Matlab PDEToolbox} (version 2.1) commands \texttt{initmesh} and \texttt{'hmax'} set to $0.2$ and  subsequent nested meshes are obtained by regular diadic refinement using \texttt{refinemesh}. We use our own code to assemble the matrices $\matriz{K}$, $\matriz{M}$ and $\matriz{B}$.

For the pcg we use as initial guess the null vector, a relative tolerance of 1E-10 and the preconditioner  $\matriz{P} = \matriz{M}_{\indice{B},\indice{B}}+\nu\matriz{B}_{\indice{B},\indice{B}}$. At each iteration we have to solve the linear systems \eqref{E200} and \eqref{E300}. To do this, we first obtain the Cholesky factors of $\matriz{K}_{\indice{I},\indice{I}}$ so that at each iteration we only have to solve triangular systems. The interior mesh nodes were ordered using a symmetric approximate minimum degree permutation in order to minimize the number of nonzero elements in the Cholesky factors of $\matriz{K}_{\indice{I},\indice{I}}$ using \textsc{Matlab} command \texttt{symamd}. An analogous reordering of the boundary nodes is done to optimize the sparsity pattern of the Cholesky factors of the preconditioner. This reorderings and factorizations take more or less the same cputime  than each single iteration of the pcg. This time is included in $t_P$.

For reference and possible double-check, we also report on the optimal value of the functional. We have marked with $\infty$ the experiments in which we have run out of RAM.

From the data in Table \ref{T31} it is clear that the proposed iterative method for the reduced problem \eqref{E600} has an overall better performance for Dirichlet control problems.

\begin{table}
\tbl{Execution times for the 2D problem in Example \ref{Ex31}. $\nu=1$.}
{
\begin{tabular}{cc|r|r|r|r|r|c|c}\toprule
&$h$ &$N$ & $N_\indice{I}$& $N_\indice{B}$ & $t_D$ (s) & $t_P$ (s) & pcg & $J(\bar u_h)$\\ \midrule
&$0.2\times 2^{-4}$ &10657 & 10337 & 320 & 0.5 & 0.2 & 7 &0.3470689275  \\
&$0.2\times 2^{-5}$& 42305 & 41665 & 640 &  4.1 & 1.0 & 7 &0.3471023330 \\
&$0.2\times 2^{-6}$ &168577 & 167297 & 1280 & 46.1 & 5.9 & 7 &0.3471129922 \\
&$0.2\times 2^{-7}$ &673025 & 670465 & 2560 & 453.0  & 31.0 & 7 &0.3471163823 \\
&$0.2\times 2^{-8}$ & 2689537 & 2684417 & 5120 &$\infty$  & 177.5 & 7 &0.3471174585\\
\bottomrule
\end{tabular}
}
 \label{T31}
\end{table}

\end{example}

\begin{remark}\label{R34}
We have also coded Algorithm 1 in \cite{Herzog-Sachs2010}, though the practical comparison with the other two methods is less clear.

The main difficulty is the choice of appropriate scalar products. For instance, we have tried (using the notation in  \cite{Herzog-Sachs2010}) $\matriz{X}=\matriz{K}+\matriz{B}$, $\matriz{X}=\matriz{K}+\matriz{M}$ or  $\matriz{X}=\matriz{K}+\matriz{M}+\matriz{B}$ for the scalar product in $Y_h$ and $\matriz{Q}=\matriz{K}_{\indice{I},\indice{I}}$  for the scalar product in $Y_{h0}$, but we have observed that the number of iterations needed to achieve the prescribed accuracy grows as the number of unknowns $N$ increases. For the first three meshes in the 2D example described in Table \ref{T31} and $\nu=1$ we obtain 126, 173 and 236 iterations respectively.
\end{remark}

\begin{example}\label{Ex33}In this example we show the effect of the use of $\matriz{P}=\matriz{M}_{\indice{B},\indice{B}}+\nu\matriz{B}_{\indice{B},\indice{B}}$ as a preconditioner to solve \eqref{E600}. In this case $\nu = 0.01$, $\Omega = \{x=re^{i\theta}\in\mathbb{R}^2:\ r<1,\ 0<\theta<11\pi/12\}$, and we use a mesh family graded at the origin with parameter $\mu = 13/33\approx 0.4$; see \cite{ASW1996}.
In Table \ref{T33} we compare the number of iterations used to reach the prescribed accuracy of $1$E$-10$ without and with preconditioner. This example shows very clearly the effectiveness of this strategy of preconditioning.

\begin{table} [h!]
\tbl{Number of conjugate gradient iterations without and with preconditioner in Example \ref{Ex33}. Graded mesh and $\nu = 0.01$.}
{
\begin{tabular}{cr|r|r|c|c|c}\toprule
&$N$ & $N_\indice{I}$& $N_\indice{B}$ &No $\matriz{P}$ &$\matriz{P}$ & $J(\bar u_h)$\\ \midrule
&289   & 225  & 64    &        46  &          26 & 0.0228920077\\
&1089  & 961  &128    &        103 &          24& 0.0230219261\\
&4225  & 3969 & 256   &        214 &          24 & 0.0231321072\\
&16641 &16129 & 512   &        421 &          24 & 0.0231982941\\
&66049 & 65025 & 1024 &      733    &          26 & 0.0232287479\\ \bottomrule
\end{tabular}
}
 \label{T33}
\end{table}

\end{example}

\begin{example}\label{Ex32}In this example we show that the number of pcg iterations for each fixed $h$ is independent of the value of $\nu\geq 0$ for all $\nu<\nu_0$. We take the same domain and meshes as in the 2D problem in Example \ref{Ex31} and set $y_\Omega(x)=|x|^2$, so that $J(\bar u)>0$ even for $\nu=0$. The results are summarized in Table \ref{TEx32}.
\begin{table}[h!]
\tbl{Dependence of the number of pcg iterations w.r.t. the regularization parameter in Example \ref{Ex32}.}
{
\begin{tabular}{cr|ccccccc}\toprule
&$h\setminus\nu$ & 1E4 & 1E2 & $1$ &  1E-2 &  1E-4 & 1E-6 &0 \\ \midrule
&$0.2\times2^{-5}$ & 2 & 3 & 7 & 26 & 49 & 49 & 49 \\
&$0.2\times2^{-6}$ & 2 & 3 & 7 & 28 & 63 & 65 & 65 \\
&$0.2\times2^{-7}$& 2 & 3 & 7 & 28 & 81 & 86 & 86 \\
\bottomrule
\end{tabular}
}
\label{TEx32}
\end{table}
\end{example}

\begin{example}\label{Ex3D}\textbf{A 3D example in a polyhedron.}
Up to our best knowledge, there is not a general theory about regularity of the solutions or approximation results for 3D Dirichlet control problems posed on polyhedral domains. In \cite{Deckelnick-Gunther-Hinze2009}, the authors study a semi-discrete or variational approach to 3D problems on regular domains. Although the semi-discrete approach coincides with the full approach for unconstrained problems, we cannot profit their results since the regularity of the solutions in a smooth domain is (see \cite[Lemma 2.1]{Deckelnick-Gunther-Hinze2009}) much higher than the one we may obtain in polyhedral domains.

For this example we will take $\Omega$ the unit cube $(-0.5,0.5)^3$, $y_\Omega\equiv 1$ and $\nu = 1$. First, we obtain a regularity result for the solution of our problem.

\begin{proposition}Let $\Omega\subset\mathbb{R}^3$ be the interior of a rectangular parallelepiped and $y_\Omega\in L^p(\Omega)$ for some $3<p<+\infty$. Consider $\bar u$ the solution of problem $(P^U)$. Then, $\bar u\in W^{1-1/p}(\Gamma)$, $\bar y \in W^{1,p}(\Omega)$ and $\bar\varphi\in W^{2,p}(\Omega)$. Moreover, $\bar u\equiv 0$ on the edges and corners of $\Gamma$.
\end{proposition}
\begin{proof}Since $\bar u\in L^2(\Gamma)$, using transposition and interpolation it is clear that $y\in H^{1/2}(\Omega)$. Classical Sobolev embedding in 3D leads to $y\in L^3(\Omega)$. Since $\Omega$ is a parallelepiped and $y_\Omega$ is regular enough, \cite[Theorem 1]{Dauge1989} states that $\bar\varphi\in W^{2,3}(\Omega)$, and hence $\partial_n\bar\varphi\in \Pi_i W^{2/3,3}(\Gamma_i)$, where $\Gamma_i$, $i=1:6$, are the faces of $\Gamma$. This does not imply immediately that $\bar u$ belongs to $W^{2/3,3}(\Gamma)$ because $2/3\cdot 3=2$, which is the topological dimension of $\Gamma$, and some integral compatibility condition should be imposed on the edges; but for all $q<3$, it is true that $\bar u\in W^{1-1/q,q}(\Gamma)$, and therefore $\bar y\in W^{1,q}(\Omega)$. Using again Sobolev embeddings, we have that $\bar y \in L^s(\Omega)$ for all $s<+\infty$ and hence $\bar y-y_\Omega\in L^p(\Omega)$. Applying once again \cite[Theorem 1]{Dauge1989}, we have that $\bar\varphi\in W^{2,p}(\Omega)$.

Now we have that $\partial_n\bar\varphi\in \Pi_i W^{1-1/p,p}(\Gamma_i)$ and we can prove that if we define $\partial_n\bar\varphi =0 $ on the edges of $\Omega$, then we obtain a continuous function.
To do this,
we use a similar argument to the one used in \cite[Section 4]{Casas-Gunther-Mateos2011} for 2D problems. Since $p>3$, $\bar\varphi\in C^1(\bar\Omega)$. Consider two faces $A$ and $B$ with a common edge $AB$ and let $\tau^1_X$, $\tau^2_X$ be two linearly independent vectors tangent to face $X$ ($X=A$ or $X=B$) such that $\tau^1_A=\tau^1_B$ is also tangent to the edge $AB$. Since $\bar\varphi = 0$ on $\Gamma$, we have that for every $x\in A$, $\nabla\bar\varphi(x)\cdot\tau^1_A = \nabla\bar\varphi(x)\cdot\tau^2_A = 0$ and for every $x\in B$,
$\nabla\bar\varphi(x)\cdot\tau^1_B = \nabla\bar\varphi(x)\cdot\tau^2_B = 0$. So for $x\in AB$, we have that $\nabla\bar\varphi(x)=0$  and therefore $\nabla\bar\varphi(x)\cdot n$ can be seen as a continuous function  if we set its value to $0$ on the edges, despite the jump discontinuities of the normal vector $n$.

So $\bar u$ is continuous and hence $\bar u\in W^{1-1/p,p}(\Gamma)$. The regularity of the optimal state follows from the trace theorem; see e.g. \cite{nevcas2012direct}.
\end{proof}
Using this regularity, an error estimate can be proved for the example problem.
\begin{proposition}Let $\Omega$ be a rectangular parallelepiped and $y_\Omega\in L^p(\Omega)$ for some $3<p<+\infty$. Consider $\bar u$ be the solution of problem $(P^U)$ and $\bar u_h$ be the solution of $(P^U_h)$. Then there exists some $h_0>0$ and $C>0$ such that for all $0<h<h_0$
\[\|\bar u-\bar u_h\|_{L^2(\Gamma)}\leq C h^{1-1/p}.\]
\end{proposition}
The proof follows the same guidelines as those of \cite{Casas-Raymond2006} or \cite[Section 6.1]{Mateos-Neitzel2015} and thus will be omitted. An interesting remark is that for 2D problems, we can deduce uniform convergence for the controls using the $L^2(\Gamma)$ estimate and an inverse inequality, since the boundary is 1D. This does not work for 3D problems with the error estimate at hand, since now $\Gamma$ is 2D.

To solve the problem we have used a regular mesh of identical cubes of size $h$, each of them divided into 6 tetrahedra according to the Kuhn triangulation of the cube (see e.g. \cite{Plaza2007}). Up to our best knowledge, the current version of the {\sc Matlab} {\sc PDEToolbox} (version 2.1) computes an approximation of $\matriz{B}$ using the barycenter formula. Although this does not affect the order of convergence for the FEM, this matrix plays a central role in Dirichlet control problems (for instance, $\matriz{B}_{\indice{B},\indice{B}}$ is not singular, but the barycenter approximation may be singular), so we have computed it in an exact way (using the mid-sides formula). Mesh data, computation times and optimal values are displayed in Table \ref{T32}.

\begin{table}
\tbl{Execution times for the 3D problem in Example \ref{Ex31}. $\nu=1$.}
{
\begin{tabular}{cl|r|r|r|r|r|c|c}\toprule
&$h$&$N$ & $N_\indice{I}$& $N_\indice{B}$ & $t_D$ (s) & $t_P$ (s) & pcg & $J(\bar u_h)$\\ \midrule
&$2^{-4}$&4913 &   3375 &   1538 &   0.7     &    0.2 &   6 & 0.4142332683\\
&$2^{-5}$& 35937 &  29791 &   6146 &  123.2s      &    5.2 &   6 & 0.4159847757\\
&$2^{-6}$&  274625 &  250047 &  24578 &  $\infty $     &  183.7 &   6 & 0.4164610762 \\
\bottomrule
\end{tabular}
}
 \label{T32}
\end{table}

\end{example}
\section{Control constrained problems}
\label{S4}
Problem $(P^C)$ has a unique solution $\bar u\in L^2(\Gamma)$ that satisfies $J'(\bar u)(u-\bar u)\geq 0$ for all $u\in U_{\alpha,\beta}$.
For every $0<h<h_0$, problem $(P^C_h)$ has also a unique solution $\bar u_h$ that satisfies $J'_h(\bar u_h)(u_h-\bar u_h)\geq 0$ for all $u_h\in U^h_{\alpha,\beta}$. The problems being convex, these conditions are also sufficient. Moreover it is known that $\bar u_h\to\bar u$ strongly in $L^2(\Gamma)$ and also error estimates are available in terms of the mesh size $h$ when $\nu>0$ and the domain is convex in 2D. See \cite{Casas-Raymond2006,Casas-Sokolowski2010}. In \cite{Deckelnick-Gunther-Hinze2009} the  smooth (2D and 3D) case is treated using variational approach, whose optimization process is different from the one we are presenting in the work at hand: {in \cite[Ch. 2.1]{Gunther2010}, the problem is solved using a fixed point method, convergent for $\nu$ large enough; see \cite{Hinze-Vierling2012} and Remark \ref{R41} below for the convergence of the semismooth Newton method.}
\subsection{Continuous problem}
We can formulate first order optimality conditions as: there exist unique $\bar u\in L^2(\Gamma)$, $\bar y\in H^{1/2}(\Omega)$ and $\bar\varphi\in H^1_0(\Omega)$ such that
\begin{subequations}
\begin{eqnarray}
-\Delta\bar y &=& 0\phantom{xxxx}\mbox{ in }\Omega,\ \bar y=\bar u\mbox{ on }\Gamma,\label{E800}\\
-\Delta\bar \varphi&=&\bar y-y_\Omega\mbox{ in }\Omega,\ \bar\varphi=0\mbox{ on }\Gamma,\label{E900}\\
(-\partial_n\bar\varphi  + \nu\bar u, u-\bar u)&\geq& 0\phantom{xxxx}\ \mbox{ for all }u\in U_{\alpha,\beta}.\label{E950}
\end{eqnarray}

We first describe a semismooth Newton method to solve this optimality system and prove a convergence result for it (see Theorem \ref{T4.2}). Next, we will reformulate the optimality system in terms of the Lagrange multipliers related to the constraints. We will see that this approach is better suited to the discrete problem, but has the drawback that the Newton method related to it is not semismooth; nonetheless, we will prove a convergence result for it; cf. Theorem \ref{T43}.
To facilitate the notation, we will skip the lower bound $\alpha$, and will work only with the constraint $u\leq \beta.$

The variational inequality \eqref{E950} is a projection in $L^2(\Gamma)$. In this case, it is equivalent to a pointwise projection:
\begin{equation}
\nu \bar u(x)=\min\{\nu \beta(x),\partial_n\bar\varphi(x)\}\mbox{ on }\Gamma.\label{E1000}
\end{equation}
\end{subequations}
In order to analyze the semismooth Newton method to solve the optimality system, we define $G:L^2(\Gamma)\to L^2(\Gamma)$ by
\[G(u) = \nu u-\min\{\nu \beta,-S^*(S u-y_\Omega)\}.\]
Thus, solving the optimality system is equivalent to solving
\begin{equation}\label{E1030}G(u)=0.\end{equation}
Given $u\in L^2(\Gamma)$, we define the sets of active and free points related to $u$ as
\[\Gamma_{{A}}(u)=\{x\in\Gamma:\ -S^*(S u-y_\Omega)>\nu \beta \},\ \Gamma_{{F}}(u) = \Gamma\setminus\Gamma_{{A}}(u).\]
Abusing notation and when this does not lead to confusion, we will often drop the $u$ and only write $\Gamma_A$ and $\Gamma_F$.  (Although it is customary to use the word ``inactive'', we have preferred to use ``free'' since we have already used the letter $I$ for the identity matrix $\matriz{I}$ and the interior nodes $\indice{I}$.) $\chi_B$ will denote the characteristic function of a set $B$.
\begin{lemma}\label{L4.1}The relation $G:L^2(\Gamma)\to L^2(\Gamma)$
is slantly differentiable in the sense stated in \cite{Hintermuller-Ito-Kunisch2003} and  $\partial^{CL}G$ semismooth in the sense stated in \cite[Ch. 2]{Hinze_Pinnau_Ulbrich_Ulbrich2009}, where $\partial^{CL}G$ is Clarke's generalized differential \cite{Clarke1983}.

A slanting functional $M(u)\in\mathcal{L}(L^2(\Gamma),L^2(\Gamma))$ is given by
\[M(u)v = \nu v+\chi_{\Gamma_{F}(u)}S^*Sv\]
for all $v\in L^2(\Gamma)$.

Finally, if $\nu>0$, $M(u)$ has an inverse uniformly bounded in $\mathcal{L}(L^2(\Gamma),L^2(\Gamma))$ for all $u\in L^2(\Gamma)$.
\end{lemma}
\begin{proof}Using  \eqref{E2.1} in Lemma \ref{L2.1}, we have  that there exists some $q>2$ such that $S^*S\in \mathcal{L}(L^2(\Gamma),L^q(\Gamma))$. The slant differentibility then follows directly from \cite[Pr 4.1 b)]{Hintermuller-Ito-Kunisch2003} and the semismoothness from \cite[Theorem 2.13]{Hinze_Pinnau_Ulbrich_Ulbrich2009};  see also \cite{Ulbrich2003}.

The expression for the slanting functional follows from the slant derivative of the function $\min(0,\cdot)$ and the chain rule \cite[Theorem 2.10(c)]{Hinze_Pinnau_Ulbrich_Ulbrich2009}.

Finally, given $z\in L^2(\Gamma)$
\[M(u)v = z\iff\left\{\begin{array}{ll}
\nu v = z & \mbox{ on }\Gamma_{A}(u),\\
\nu v =  z -S^*Sv& \mbox{ on }\Gamma_{F}(u).\end{array}\right.\]
This equations can be read as
\begin{eqnarray*}
-\Delta y  =0\mbox{ in }\Omega, & &y =v\mbox{ on }\Gamma, \\
-\Delta \varphi  = y\mbox{ in }\Omega,& & \varphi =0\mbox{ on }\Gamma, \\
\nu v = z  \mbox{ on }\Gamma_{A}(u),&&\nu v = z + \partial_n\varphi \mbox{ on }\Gamma_{F}(u).
\end{eqnarray*}
Let us define $w=\chi_{\Gamma_{F}(u)}v$. Taking into account the definition of $S$ and $S^*$, we have that
$\nu v = \nu w +  z\chi_{\Gamma_{A}(u)}$ and
\begin{align*}
-\Delta y_w  =0\mbox{ in }\Omega,\qquad &y_w =w\mbox{ on }\Gamma,   \\
-\Delta \varphi_w  = y_w\mbox{ in }\Omega,\qquad & \varphi_w =0\mbox{ on }\Gamma, \\
w = 0  \mbox{ on }\Gamma_{A}(u),\qquad&
\nu w = z + \partial_n\varphi_w -  S^*S z \chi_{\Gamma_{A}(u)}\mbox{ on }\Gamma_{F}(u).
\end{align*}
On one hand, using the continuity of $S^*$ (cf. \eqref{E2.1}) and of $S$, we have
\begin{equation}
\| S^*S z \chi_{\Gamma_{A}(u)}\|_{L^2(\Gamma)}\leq C  \|z\|_{L^2(\Gamma_{A}(u))}\leq C \|z\|_{L^2(\Gamma)}.
\end{equation}
On the other hand, using the definition of solution in the transposition sense, we have that
\[(y_w,y_w)_\Omega = -(w,\partial_n\varphi_w)_\Gamma = - \nu (w,w)_{\Gamma} +  (w,z)_{\Gamma} -  (w,S^*Sz \chi_{\Gamma_{A}(u)})_{\Gamma}.\]
So
\begin{eqnarray*}
\nu \|w\|^2_{L^2(\Gamma)} &=&  (w,z)_{\Gamma} - (w,S^*S z \chi_{\Gamma_{A}(u)})_{\Gamma}- \|y_w\|^2_{L^2(\Omega)}\\
&\leq & \|w\|_{L^2(\Gamma)} \|z\|_{L^2(\Gamma)} + \|w\|_{L^2(\Gamma)}\|S^*Sz \chi_{\Gamma_{A}(u)}\|_{L^2(\Gamma)} \\ & \leq& C\|w\|_{L^2(\Gamma)} \|z\|_{L^2(\Gamma)}.
\end{eqnarray*}
And we get $\|w\|_{L^2(\Gamma)}\leq C/\nu  \|z\|_{L^2(\Gamma)}$. Taking into account the definition of $w$ and the condition for $v$ on the active set, we get that $\|v\|_{L^2(\Gamma)}\leq C/\nu \|z\|_{L^2(\Gamma)}$, where $C$ is independent of $u$, and hence $M(u)$ has a uniformly bounded inverse for each $\nu>0$.
\end{proof}
\begin{remark}Notice that in the infinite dimensional case, if $\nu=0$ then $M(u)v = \chi_{\Gamma_{F}(u)}S^*Sv$. In practical cases, it is known (cf. \eqref{E2.1a} or \cite{AMPR2015}) that there exists $t>0$ such that $S^*Sv\in H^t(\Gamma)$ which is compactly embedded in $L^2(\Gamma)$, and hence $M(u)$ does not have a bounded inverse.
\end{remark}

Given a current iterate $u\in L^2(\Gamma)$, we may compute next iterate $u^+$ using a semismooth Newton method as follows:
\[M(u)(u^+-u) = -G(u).\]
Writing this in detail leads to
\[\nu u^+-\nu u +\chi_{\Gamma_{F}(u)}S^*S(u^+-u) = -\nu u+\min\{\nu \beta,-S^*(Su-y_\Omega)\},\]
which means that, if $\nu>0$,
\[\nu u^+-\nu u = -\nu u+\nu \beta\Rightarrow u^+ = \beta\mbox{ on }\Gamma_{{A}}(u),\]
and
\[u^+ = -\frac{1}{\nu}S^*(S u^+-y_\Omega)\mbox{ on }\Gamma_{F}(u).\]
This equations can be read as
\begin{subequations}
\begin{eqnarray}
-\Delta y^+ = 0\mbox{ on }\Omega, && y^+ = u^+\mbox{ on }\Gamma,\label{E4.15}\\
-\Delta \varphi^+ = y^+-y_\Omega\mbox{ on }\Omega, && \varphi^+ = 0\mbox{ on }\Gamma,\label{E4.16}\\
u^+ = \beta\mbox{ on }\Gamma_{A}(u), &&
u^+ = \frac{1}{\nu}\partial_n\varphi^+\mbox{ on }\Gamma_{F}(u).\label{E4.17}
\end{eqnarray}
\end{subequations}
With all these considerations, we can write a semismooth Newton method to solve the optimality system in the infinite dimensional case.

\LinesNumbered
\begin{algorithm}[H]
\caption{Semismooth Newton method to solve \eqref{E800}, \eqref{E900}, \eqref{E1000}\label{A300}}
\DontPrintSemicolon
Initialize $k=0$ and provide $u^0\in L^2(\Gamma)$\;
Set $u = u^k$\;\label{L4.2}
Compute $\Gamma_{A}(u)$ and $\Gamma_{F}(u)$\;
Solve the optimality system \eqref{E4.15} \eqref{E4.16}, \eqref{E4.17}\;
Set $u^{k+1}=u^+$\;
Set $k=k+1$\;
Stop or go to \ref{L4.2}
\end{algorithm}

\begin{theorem}\label{T4.2}The semismooth Newton method described in Algorithm \ref{A300} converges q-superlinealy to $\bar u$ provided $u^0$ is close enough to $\bar u$ in the sense of $L^2(\Gamma)$.
\end{theorem}
\begin{proof}Once we have proved Lemma \ref{L4.1}, this result is a direct consequence of \cite[Theorem 3.4]{Chen-Nashed-Qi2000}; see also \cite[Theorem 1.1]{Hintermuller-Ito-Kunisch2003} or \cite[Theorem 2.12]{Hinze_Pinnau_Ulbrich_Ulbrich2009}.
 \end{proof}

The discrete version of the variational inequality \eqref{E950} --see \eqref{E418d} below-- does not have a pointwise version analog to \eqref{E1000}. A more convenient approach to the continuous problem, from the point of view of the discretized problem, is obtained using Lagrange multipliers associated to the bound control constraints.  The con of this approach is that we do not obtain a semismooth Newton method, in the sense that the involved functions are known not to be semismooth.
\begin{remark}\label{R41}
 {It must noticed that, when the variational discretization of the control is used, \cite{Hinze2005,Deckelnick-Gunther-Hinze2009,Gunther2010,Hinze-Vierling2012}, the discrete optimal control is obtained as the pointwise projection of the discrete optimal adjoint state as in \eqref{E1000}, so convergence of the semismooth Newton method for such a discretization would follow from \refre{ Theorem 4.2  and Lemma \ref{L4.1} with $S$ replaced by $S_h$}.}
\end{remark}

To simplify the notation, we will reduce the exposition to the case of having only an upper bound $u\leq \beta$ on $\Gamma$.

Condition \eqref{E950} can be replaced by the following pair of equations: there exists also $\bar\lambda\in L^2(\Gamma)$ such that equation \eqref{E1000} can be written as
\begin{subequations}
\begin{eqnarray}
\nu\bar u(x)&=&\partial_n\bar\varphi(x)-\bar\lambda(x)\mbox{ on }\Gamma.\label{E1100}\\
\bar u(x)\leq\beta(x),\ \bar \lambda(x)\geq 0,\ \bar \lambda(x)(\bar u(x)-\beta(x))&=&0\mbox{ for a.e. }x\in\Gamma.\label{E1200}
\end{eqnarray}
\end{subequations}
Condition \eqref{E1200} can be written as
\begin{equation}\label{E1300}
\lambda(x)-\max\{0,\lambda(x)+c(u(x)-\beta(x))\} =0\mbox{ for a.e. }x\in\Gamma,
\end{equation}
which is true for any $c>0$.
Since $u\in L^2(\Gamma)$ appears inside the $\max$ operation, it is known that the operator described in \eqref{E1300} is not semismooth (following \cite[Lemma 2.7]{Hinze_Pinnau_Ulbrich_Ulbrich2009}) or has not a slant derivative (as in \cite[Pr 4.1]{Hintermuller-Ito-Kunisch2003}), so it is not clear that  Newton's method applied to \eqref{E800}, \eqref{E900}, \eqref{E1300} converges.
Define
\[F_c(u,\lambda) = \left(\begin{array}{c}
\displaystyle{\nu u+S^*Su-S^*y_\Omega+\lambda}\\
\displaystyle{\lambda-\max\{0,\lambda+cu-c\beta\}}
\end{array}
\right)
\]
and for each pair control-multiplier $(u,\lambda)\in L^2(\Gamma)\times L^2(\Gamma)$, set
\[\Gamma_{A}(u,\lambda) = \{x\in\Gamma:\ \frac{1}{c}\lambda+u>\beta\},\ \Gamma_{F}(u,\lambda) = \Gamma\setminus \Gamma_{A}(u,\lambda).\]
We have that
\[ \left(\begin{array}{cc}
\nu I+S^*S & I \\
- c\chi_{\Gamma_{A}(u,\lambda)} I & \chi_{\Gamma_{F}(u,\lambda)} I \end{array}
\right)\in \partial^{CL} F_c(u,\lambda).
\]
So, given a current iterate $(u,\lambda)$, a Newton-like iterate to obtain $(u^+,\lambda^+)$ reads as
\begin{subequations}
\begin{eqnarray}
\nu u^+ = -S^*(Su^+-y_\Omega) - \lambda^+,\label{E4.19a}\\
u^+ = \beta \mbox{ on } \Gamma_{A}(u,\lambda),\label{E4.19b}\\
\lambda^+ =0 \mbox{ on } \Gamma_{F}(u,\lambda).\label{E4.19c}
\end{eqnarray}
\end{subequations}
It is remarkable that the parameter $c$ only appears in this equations hidden in the definition of the active set.

\LinesNumbered
\begin{algorithm}[H]
\caption{Newton-like method to solve \eqref{E800}, \eqref{E900}, \eqref{E1300}\label{A400a}}
\DontPrintSemicolon
Initialize $k=0$ and provide $(u^0,\lambda^0)\in L^2(\Gamma)\times L^2(\Gamma)$\;
Set $u = u^k$ and $\lambda = \lambda^k$\;\label{L4.2bis}
Compute $\Gamma_{A}(u,\lambda)$ and $\Gamma_{F}(u,\lambda)$\;
Compute $(u^+,\lambda^+)$ using  \eqref{E4.19a} \eqref{E4.19b}, \eqref{E4.19c}\;
Set $u^{k+1}=u^+$ and $\lambda^{k+1}=\lambda^+$\;
Set $k=k+1$\;
Stop or go to \ref{L4.2bis}
\end{algorithm}

\begin{theorem}\label{T43}If $c=\nu>0 $, the sequence $u^k$ generated by the Newton-like method described in Algorithm \ref{A400a} converges q-superlinearly to $\bar u$ in $L^2(\Gamma)$ provided $u^0$ is close enough to $\bar u$ in the sense of $L^2(\Gamma)$ and $\lambda^0=S^*y_\Omega-(S^*Su^0 + \nu u^0)$.
\end{theorem}
\begin{proof}In this case we cannot apply directly the results in \cite[Theorem 3.4]{Chen-Nashed-Qi2000}; see also \cite[Theorem 1.1]{Hintermuller-Ito-Kunisch2003} or \cite[Theorem 2.12]{Hinze_Pinnau_Ulbrich_Ulbrich2009} since $u\in L^2(\Gamma)$ appears inside the $\max$ operation, and it is known that the operator described in \eqref{E1300} is not semismooth (following \cite[Lemma 2.7]{Hinze_Pinnau_Ulbrich_Ulbrich2009}) or has not a slant derivative (as in \cite[Pr 4.1]{Hintermuller-Ito-Kunisch2003}).

We will follow instead the method of proof of \cite[Theorem 4.1]{Hintermuller-Ito-Kunisch2003} and we will show that the sequence $u^k$ generated by Algorithm \ref{A400a} is exactly the same as the one generated by Algorithm \ref{A300}, which we will call $\tilde u^k$.

Writing in detail \eqref{E4.19a} we obtain
\begin{subequations}
\begin{eqnarray}
-\Delta y^+ = 0 \mbox{ on }\Omega, && y^+ = u^+\mbox{ on }\Gamma,\label{E4.20a}\\
-\Delta \varphi^+ = y^+-y_\Omega\mbox{ in }\Omega, && \varphi^+ = 0\mbox{ on }\Gamma,\label{E4.20b}\\
&& \nu u^+ = \partial_n\varphi^+ -\lambda^+\mbox{ on }\Gamma.\label{E4.20c}
\end{eqnarray}
\end{subequations}
Using now \eqref{E4.19b} and \eqref{E4.19c} we obtain
\begin{subequations}
\begin{eqnarray}
-\Delta y^+ = 0\mbox{ on }\Omega, && y^+ = u^+\mbox{ on }\Gamma,\label{E4.15bis}\\
-\Delta \varphi^+ = y^+-y_\Omega\mbox{ on }\Omega, && \varphi^+ = 0\mbox{ on }\Gamma,\label{E4.16bis}\\
u^+ = \beta\mbox{ on }\Gamma_{A}(u,\lambda), &&
u^+ = \frac{1}{\nu}\partial_n\varphi^+\mbox{ on }\Gamma_{F}(u,\lambda),\label{E4.17bis}\\
\lambda^+ = \partial_n\varphi^+-\nu \beta  \mbox{ on }\Gamma_{A}(u,\lambda), &&\lambda^+=0\mbox{ on }\Gamma_{F}(u,\lambda).\label{E4.23d}
\end{eqnarray}
\end{subequations}
Notice that  \eqref{E4.15bis}, \eqref{E4.16bis} and \eqref{E4.17bis} is exactly as \eqref{E4.15}, \eqref{E4.16} and \eqref{E4.17}, provided that $\Gamma_{A}(u,\lambda)= \Gamma_{A}(u)$.

To finish, let us prove by induction that $u^k=\tilde u^k$ and $\Gamma_{A}(u^k,\lambda^k)= \Gamma_{A}(\tilde u^k)$ for all $k\in\mathbb{N}\cup\{0\}$. For $k=0$ it is clear from the definition of the active sets and the choice of $\lambda^0$ made in the assumption of this theorem. Suppose now $u^k=\tilde u^k$ and $\Gamma_{A}(u^k,\lambda^k)= \Gamma_{A}(\tilde u^k)$. It is clear that $u^{k+1}=\tilde u^{k+1}$. On the other hand, from \eqref{E4.17bis} and \eqref{E4.23d} and the choice $c=\nu$ we have that $c u^{k+1}+\lambda^{k+1}= - S^*S(u^{k+1}-y_\Omega)$, and therefore $\Gamma_{A}(u^{k+1},\lambda^{k+1})= \Gamma_{A}(\tilde u^{k+1})$.

The result, hence, follows from Theorem \ref{T4.2}.
\end{proof}
We want to remark here that writting the Newton step as in \eqref{E4.15bis}, \eqref{E4.16bis}, \eqref{E4.17bis} and \eqref{E4.23d} we obtain exactly the same kind of algorithm as the one described in \cite{Kunisch-Rosch2002}

\subsection{Finite dimensional approximation}
Let us focus now on the finite dimensional approximation.
For every $0<h<h_0$, there exist a unique solution $\bar u_h\in U_h$ of $(P_h^{C})$ and unique, $\bar y_h\in Y_h$ and $\bar\varphi_h\in Y_{h0}$ such that
\begin{subequations}
\begin{eqnarray}
a(\bar y_h, z_h)&=&0 \mbox{ for all }z_h\in Y_{h0},\label{E418a}\\
 \bar y_h&\equiv& \bar u_h\mbox{ on }\Gamma,\label{E418b}\\
a( z_h,\bar\varphi_h) & =&(\bar y_h-y_\Omega,z_h)_\Omega \mbox{ for all }z_h\in Y_{h0},\label{E418c}\\
(-\partial_n^h\bar\varphi_h+\nu\bar u_h,u_h-\bar u_h) & \geq&0\mbox{ for all }u_h\in U^h_{\alpha,\beta}.\label{E418d}
\end{eqnarray}
\end{subequations}
This way of writting the optimality system is useful to obtain error estimates (cf. \cite{Casas-Raymond2006,Deckelnick-Gunther-Hinze2009,Mateos-Neitzel2015}). Nevertheless, we cannot deduce a pointwise projection formula for the optimal control from the variational inequality.
\begin{remark}\label{naive}Think of this naive example. Let $\Gamma=[-1,1]$ and take a mesh with nodes $\{-1,0,1\}$. For every $u_h,v_h\in U_h$, we have that
$(u_h,v_h)_\Gamma = \bm{u}^T \matriz{B}\bm{v}$, where $\matriz{B}$ is the mass matrix
\[\matriz{B}=\frac{1}{6}\left(\begin{array}{ccc}
2&1&0\\1&4&1\\0&1&2\end{array}\right).\]
Consider $\alpha = -\infty$, $\beta = 0$, and take , e.g., $\bm{w} =(-2,1,1)^T$. Then we have that
\[(-w_h+\bar u_h,u_h-\bar u_h)\geq 0\ \forall u_h\in U^h_{\alpha,\beta} \implies
\bar u_h = -1.5e_1 + 0 e_2 + 0 e_3\implies \bm{u} = (-1.5,0,0)^T,\]
but
\[\min(\bm{w},\bm{0})= (-2,0,0)^T,\]
and we get different results with the $L^2(\Gamma)$ projection and the pointwise projection.
\end{remark}
To circumvent this difficulty, we rewrite conditions \eqref{E418a}--\eqref{E418d} in order to use standard semismooth Newton method. Taking into account \eqref{E9}, we can write $(P^C_h)$ as
\begin{equation}\begin{array}{l}
\min \frac12 \bm{u}^T \matriz{A}\bm{u} -\bm{f}^T\bm{u}\\
\mbox{ subject to }\bm{\alpha}\leq \bm{u}\leq \bm{\beta},
\end{array}
\label{E1420}\end{equation}
where $\matriz{A}$ and $\bm{f}$ are defined in \eqref{E2end} and \eqref{E2.9} and $\bm{\alpha},\bm{\beta}\in\mathbb{R}^{N_{\indice{B} \times 1}}$ are the vectors whose $j$-component are respectively $\alpha(x_j)$ and $\beta(x_j)$ and the inequalities are understood componentwise. In order to simplify the notation, we will restrict ourselves again to the case without lower bound.
The optimality system can thus be written in the following form: if $\bm{u}$ is the solution of problem \eqref{E1420}, then there exists a unique $\bm{\lambda}\in \mathbb{R}^{N_{\indice{B}}\times 1}$ such that
\begin{equation}\begin{array}{rcl}
\matriz{A}\bm{u}+\bm{\lambda} -\bm{f} &=&0,\\
\bm{\lambda}\geq 0,\ \bm{u}\leq \bm{\beta},\ \bm{\lambda}^T(\bm{u}-\bm{\beta})&=&0.
\end{array}\label{E1450}
\end{equation}
Notice that although we may name $\bar\lambda_h(x) = \sum_{j\in\indice{B}} \lambda_j e_j(x)$ for $x\in \Gamma$ to obtain a Lagrange multiplier $\bar\lambda_h\in U_h$, the complementarity condition with respect to the dot product in $\mathbb{R}^{N_\indice{B}}$, does {\em not} imply that
$(\bar\lambda_h,\bar u_h-\beta)_\Gamma =0,$
which shows again the convenience of using \eqref{E1450} instead of \eqref{E418a}--\eqref{E418d}. To continue, we rewrite again the second condition in \eqref{E1450} to obtain
\begin{equation}\begin{array}{rcl}
\matriz{A}\bm{u}+\bm{\lambda} -\bm{f} &=&0,\\
\bm{\lambda}-\max(0,\bm{\lambda}+c(\bm{u}-\bm{\beta}))&=&0.
\end{array}\label{E1500}
\end{equation}
Notice that \eqref{E1500} looks like a discrete version of the optimality system formed by \eqref{E800}, \eqref{E900} and \eqref{E1300}.
Acting in an analogous way as we did for the continuous problem, we define
\[F_{h,c}(\bm{u},\bm{\lambda}) = \left(\begin{array}{c}
\matriz{A}\bm{u}+\bm{\lambda} -\bm{f}\\
\bm{\lambda}-\max(\bm{0},\bm{\lambda}+c(\bm{u}-\bm{\beta}))
\end{array}\right).
\]
and for every pair $(\bm{u},\bm{\lambda})$, we define the
sets of active  and free
  indexes as
\begin{equation}\label{E4.14}
\indice{A}(\bm{u},\bm{\lambda}) = \{j:\lambda_j+c(u_j-\beta(x_j))> 0\},\
\indice{F}(\bm{u},\bm{\lambda}) = \{j:\lambda_j+c(u_j-\beta(x_j))\leq 0\}.\end{equation}
Abusing notation and when this does not lead to confusion, we will often just write $\indice{A}$ and $\indice{F}$. Notice that $\indice{A}\cup\indice{F}=\indice{B}$ for every possible pair control-multiplier.
\begin{lemma}\label{L4.4}The function $F_{h,c}$ is slantly differentiable. A slanting function for $F_{h,c}$ is
\[M_{h,c}(\bm{u},\bm{\lambda}) = \left(
\begin{array}{cc}
\matriz{A} & \matriz{I}_{\indice{B},\indice{B}}\\
-c \matriz{I}_{\indice{B},\indice{A}} \matriz{I}_{\indice{A},\indice{B}} & \matriz{I}_{\indice{B},\indice{F}} \matriz{I}_{\indice{F},\indice{B}}\end{array}
\right).
\]
Finally, for all $\nu\geq 0$, the inverse of $M_{h,c}(\bm{u},\bm{\lambda}) $ is uniformly bounded w.r.t $(\bm{u},\bm{\lambda})$.
\end{lemma}
\begin{proof}
In finite dimension, the function $\max\{\bm{0},\cdot\}$ is semismooth due to Rademacher's theorem (see e.g. \cite[Example 2.4]{Hinze_Pinnau_Ulbrich_Ulbrich2009} or \cite[Lemma 3.1]{Hintermuller-Ito-Kunisch2003}).

A forward computation shows that $M_{h,c}$ is a slanting function.

Let us prove the uniform boundedness of the inverse. Given $\bm{z}\in \mathbb{R}^{N_{\indice{B}}}$ and $\bm{\eta}\in \mathbb{R}^{N_{\indice{B}}}$, we have that
\[M_{h,c}(\bm{u},\bm{\lambda}) \left[\begin{array}{c}\bm{v}\\ \bm{\mu}\end{array}\right] =
\left[\begin{array}{c}\bm{z}\\ \bm{\eta}\end{array}\right]\iff\left\{\begin{array}{rcl}
\matriz{A}\bm{v}+\bm{\mu} & = & \bm{z},\\
-c\bm{v}_\indice{A} & = &\bm{\eta}_\indice{A},\\
\bm{\mu}_\indice{F} & = & \bm{\eta}_\indice{F}.\end{array}\right.
\]
We write the last equalities as
\begin{eqnarray*}
\matriz{A}\bm{v}+
\left[
\begin{array}{c}
\bm{\mu}_\indice{A} \\ \bm{0}_\indice{F}
\end{array}
\right]
& = &
\bm{z}-
\left[
\begin{array}{c}\bm{0}_\indice{A}\\ \bm{\eta}_\indice{F}
\end{array}
\right] =:\bm{\zeta}, \\
-c\bm{v}_\indice{A} & = &\bm{\eta}_\indice{A}.
\end{eqnarray*}
This is the optimality system of the  equality constrained optimization problem
\begin{eqnarray*}
&&\bm{v} = \arg\min_{\bm{w}\in \mathbb{R}^{N_{\indice{B}}}} \frac{1}{2}\bm{w}^T\matriz{A}\bm{w} - \bm{\zeta}^T\bm{w},\\
&&\text{subject to }\bm{w}_\indice{A} = -{\bm{\eta}_\indice{A}}/{c}.
\end{eqnarray*}
Writing $\bm{w} = \matriz{I}_{\mathcal{B},\mathcal{F}}\bm{w}_\indice{F} + \matriz{I}_{\mathcal{B},\mathcal{A}}\bm{w}_\indice{A} $ and taking into account the equality constraint, we have that $\bm{v}_\indice{F}$ is the solution of the following unconstrained optimization problem

\[ \bm{v}_\indice{F} = \displaystyle\arg\min_{\bm{w}_\indice{F}} \frac12 \bm{w}_\indice{F}^T
(\matriz{I}_{\indice{F},\indice{B}}\matriz{A}\matriz{I}_{\indice{B},\indice{F}})
\bm{w}_\indice{F} -
\big(\matriz{I}_{\indice{F},\indice{B}}( \bm{\zeta}  + \matriz{A} \matriz{I}_{\indice{B},\indice{A}} \bm{\eta}_\indice{A}/c )\big)^T
\bm{w}_\indice{F}\]
and therefore $\bm{v}_\indice{F}$ is the solution of the following linear system
\[\matriz{I}_{\indice{F},\indice{B}}\matriz{A}\matriz{I}_{\indice{B},\indice{F}}
\bm{v}_\indice{F} = \matriz{I}_{\indice{F},\indice{B}}( \bm{\zeta} +\matriz{A} \matriz{I}_{\indice{B},\indice{A}} \bm{\eta}_\indice{A} /c)\]
Since $\matriz{A}$ is symmetric and positive definite, so is $\matriz{I}_{\indice{F},\indice{B}}\matriz{A}\matriz{I}_{\indice{B},\indice{F}}$, and its smallest eigenvalue is bounded from below by $0<\lambda_1(\matriz{M})+\nu\lambda_1(\matriz{B}_{\indice{B},\indice{B}})$; see  \eqref{E35}. Therefore, the previous system is solvable and there exists a constant $C>0$, that may depend on $h$, $\nu$ and $c$, but is independent of $\bm{u}$,  such that $\|\bm{v}_\indice{F}\|\leq C (\|\bm{z}\|+\|\bm{\eta}\|)$. From this it is straight to deduce that  $\|\bm{v}\|+ \|\bm{\mu}\|\leq C (\|\bm{z}\|+\|\bm{\eta}\|)$ and hence $M_{h,c}$ has a uniformly bounded inverse.
\end{proof}

With these considerations, given a current iterate $(\bm{u},\bm{\lambda})$ with active and free index sets $\indice{A}=\indice{A}(\bm{u},\bm{\lambda})$ and $\indice{F}=\indice{F}(\bm{u},\bm{\lambda})$, we can compute the next step of Newton's method
$(\bm{u}^+,\bm{\lambda}^+)$ solving
\[\begin{array}{l}
\bm{u}^+ = \displaystyle\arg\min_{\bm{u}\in\mathbb{R}^{N_\indice{B}\times 1}} \frac12 \bm{u}^T \matriz{A}\bm{u} -\bm{f}^T\bm{u}\\
\phantom{xxxx}\mbox{ subject to }\bm{u}_\indice{A}= \bm{\beta}_\indice{A}.\\
\mbox{ Set }\bm{\lambda}^+ =  \bm{f} -\matriz{A}\bm{u}^+.
\end{array}
\]
At each iteration, this is equivalent to solving the following unconstrained optimization problem in the lower-dimensional space
$\mathbb{R}^{N_{\indice{F}} \times 1}$:
\begin{equation}\begin{array}{l}
 \bm{u}^+_\indice{F} = \displaystyle\arg\min_{\bm{u}_\indice{F}
\in\mathbb{R}^{N_\indice{F}\times 1}} \frac12 \bm{u}_\indice{F}^T
(\matriz{I}_{\indice{F},\indice{B}}\matriz{A}\matriz{I}_{\indice{B},\indice{F}})
\bm{u}_\indice{F} -
\big(\matriz{I}_{\indice{F},\indice{B}}( \bm{f} -\matriz{A} \matriz{I}_{\indice{B},\indice{A}} \bm{\beta}_\indice{A} )\big)^T
\bm{u}_\indice{F} \\
\bm{u}^+_\indice{A}= \bm{\beta}_\indice{A}
\\
\bm{\lambda}^+ =  \bm{f} -\matriz{A}\bm{u}^+
\end{array}\label{E1700}\end{equation}
Again the preconditioned gradient method works fine to solve this problem. A good preconditioner in practice is $\matriz{P} = \matriz{M}_{\indice{F},\indice{F}}+\nu\matriz{B}_{\indice{F},\indice{F}}$.

Alternatively, taking into account the definition of $\matriz{A}$
 and $\bm{f}$ (see also Algorithms \ref{A100} and \ref{A200} respectively), we may write one step of the semismooth Newton algorithm as
\begin{equation}
\begin{array}{rcl}
\left[\begin{array}{ccc}
\matriz{M}+{\nu}\matriz{B} & -\matriz{K}_{:,\indice{I}}  & \matriz{I}_{:,\indice{A}}\\
-\matriz{K}_{\indice{I},:} & \matriz{O}_{\indice{I},\indice{I}} & \matriz{O}_{\indice{I},\indice{A}}\\
 \matriz{I}_{\indice{A},:} & \matriz{O}_{\indice{A},\indice{I}} & \matriz{O}_{\indice{A},\indice{A}}
\end{array}\right]
\left[\begin{array}{c}
\bm{y}^+ \\
\bm{\varphi}^+_\indice{I}\\
\bm{\lambda}^+_\indice{A} \end{array}
\right] &=&
\left[\begin{array}{c}
\matriz{M}\bm{y}_\Omega\\
\bm{0}_{\indice{I}}\\
\matriz{I}_{\indice{A},\indice{B}}\bm{\beta}
\end{array}
\right]\\
\bm{u}^+ &=& \bm{y}^+_{\indice{B}}\\
\bm{\lambda}^+_{\indice{F}} &=&\bm{0}_{\indice{F}}
\end{array}
\label{E1600}
\end{equation}
As for the unconstrained problem, direct methods for small size problems or preconditioned conjugate gradient techniques described in \cite{Herzog-Sachs2010} can be applied to solve this system at each Newton step.

\begin{remark}\label{R45}Notice that the only information from iteration $k$ used to compute iteration $k+1$ is the set of active indexes. Therefore, when $\indice{A}(\bm{u}_{k+1},\bm{\lambda}_{k+1})=\indice{A}(\bm{u}_{k},\bm{\lambda}_{k})$ we have reached an stationary point. This is usually the criterion used to stop the semismooth Newton method. Another consequence of this is that to initiate the algorithm, in principle we do not not need an initial guess $\bm{u}^0$ and $\bm{\lambda}^0$, but only an initial guess for the active set. We include nevertheless an initial guess for the control variable in Algorithm \ref{A450} because we use $\bm{u}_k$ as the initial guess for the pcg to obtain $\bm{u}_{k+1}$.
\end{remark}

\LinesNumbered
\begin{algorithm}[H]
\caption{Semismooth Newton method for $(P^C_h)$\label{A450}}
\DontPrintSemicolon
Set $k=0$ and initialize $\bm{u}_0$ and $\bm{\lambda}_0$; compute $\indice{A}_0 = \indice{A}(\bm{u}_{0},\bm{\lambda}_{0})$\;
Set $\indice{A}=\indice{A}_k$ and $\indice{F}=\indice{B}\setminus\indice{A}$\;\label{L2}
Compute $(\bm{u}_{k+1},\bm{\lambda}_{k+1}) = (\bm{u}^+,\bm{\lambda}^+)$  using \eqref{E1700} [or \eqref{E1600}]\;
Set $\indice{A}_{k+1} = \indice{A}(\bm{u}_{k+1},\bm{\lambda}_{k+1})$\;
Stop or set $k= k+1$ and return to \ref{L2}\;
\end{algorithm}

The following convergence result is a direct consequence of Lemma \ref{L4.1} and \cite[Theorem 3.4]{Chen-Nashed-Qi2000}; see also \cite[Theorem 1.1]{Hintermuller-Ito-Kunisch2003} or \cite[Theorem 2.12]{Hinze_Pinnau_Ulbrich_Ulbrich2009}.

\begin{theorem}The sequence $\bm{u}_k$ generated by Algorithm \ref{A450} converges q-superlinearly to $\bm{u}$, the solution of \eqref{E1420}.
\end{theorem}
\begin{example}We resume the 2D problem described in Example \ref{Ex31} and the 3D problem described in Example \ref{Ex3D}, with the upper constraint  $\beta\equiv0.16$. We test Algorithm \ref{A450}.

Following the tip of Theorem \ref{T43}, we have taken the parameter $c=\nu$ and $\bm{\lambda}_0=\bm{f}-\matriz{A}\bm{u}_0$. Nevertheless, we have not been able to observe any problem for different values of $c$.  The seed is set to $\bm{u}_0=\bm{0}$.

To solve the optimality system of the unconstrained optimization problem at each Newton iterate, we have used the preconditioned conjugate gradient method for \eqref{E1700} with initial guess $\matriz{I}_{\indice{F},\indice{B}}\bm{u}_k$. The use of the reduced problem is even more advisable in this case than it was in Examples \ref{Ex31} and \ref{Ex3D}, because the size of the system \eqref{E1600} is $N+N_\indice{I}+N_{\indice{A}}$, which in any case is greater or equal than the size of the system \eqref{E500}, which is $N+N_\indice{I}$.

In Tables \ref{T4.1} and \ref{T5.3} and we report on the number of Newton iterations for each mesh size as well as the total number of conjugate gradient iterations. For reference, we also report on the number of active nodes and the optimal solution of the discrete problem being approximated.

\begin{table}[h!]
\tbl{Newton iterations and total pcg iterations for a 2D control constrained problem}
{
\begin{tabular}{cc|r|r|c|c|c}\toprule
&$h$ & $N_\indice{B}$ & $N_\indice{A}$ & Newton & pcg  & $J(\bar u_h)$ \\ \midrule
&$0.2\times2^{-4}$ &320 &  245 &  3 &  16 & 0.3537665421\\
&$0.2\times2^{-5}$ & 640 &  491 &  4 &  20 & 0.3537991309\\
&$0.2\times2^{-6}$ & 1280 &  977 &  4 &  19 & 0.3538099358\\
&$0.2\times2^{-7}$ & 2560 &  1951 &  3 &  16 & 0.3538134437\\
&$0.2\times2^{-8}$ & 5120 &  3905 &  4 &  19 & 0.3538145736 \\
\bottomrule
\end{tabular}
}
\label{T4.1}
\end{table}
\begin{table}[h!]
\tbl{Newton iterations and total pcg iterations for a 3D control constrained problem}
{
\begin{tabular}{cc|r|r|c|c|c}\toprule
&$h$ & $N_\indice{B}$ & $N_\indice{A}$ & Newton & pcg  & $J(\bar u_h)$ \\\midrule
&$2^{- 2}$ &      98 &      54  & 2 &  7 & 0.3935682160\\
&$2^{- 3}$ &     386 &     294  & 2 &  7 & 0.4007301110\\
&$2^{- 4}$ &    1538 &     894  & 3 &  16 & 0.4104264396\\
&$2^{- 5}$ &    6146 &    3210  & 3 &  15 & 0.4153200584\\
&$2^{- 6}$ &   24578 &   11958  & 4 &  19 & 0.4173850169\\
\bottomrule
\end{tabular}
}
\label{T5.3}
\end{table}

\end{example}
\textbf{Choosing the initial guess for Algorithm \ref{A450}.} If possible, it is a good idea to select an initial point for the Newton method close to the solution. If we are dealing with the problem for some mesh size $h$, a good candidate for the initial iteration is $I_h\bar u_{h_{-1}}$, the solution in a coarser mesh with $h_{-1}\geq h$. This idea can be iterated with a mesh family with parameters $h_{-M},\dots,h_0=h$. The computation time should decrease for fine meshes provided that the interpolation can be carried out in an effective way. For instance, using nested meshes.
\begin{example}Solving the last 2D problem in Table \ref{T4.1} takes a cputime of 620 seconds. Using a nested iteration, we solve the problem with 5120 boundary nodes in 436  seconds, with just 2 Newton iterates and 7 conjugate gradient iterates at the finest level. The times include mesh generations and matrix assembly. To generate the meshes and interpolate the solution, we have used {\sc Matlab} {\sc PDEToolbox} command \texttt{refinemesh}.

In the 3D case (see Table \ref{T5.3}), solving the problem for $h=2^{-6}$ takes 681 seconds. Using a nested mesh strategy this time is reduced to 404 seconds with 2 Newton iterates and 9 pcg iterations at the finest level. In our example we have been able to make an efficient interpolation of the solution at the previous level just using  {\sc Matlab}'s 
\texttt{interp3}.
\end{example}
\begin{example}\label{Enu0}\textbf{Absence of Tikhonov parameter} In the same 2D pentagonal domain, we take again $\beta\equiv 0.16$, but now $\alpha \equiv -1.2$, $\nu = 0$, $y_\Omega \equiv 1$ if $x_1>0.25$, and $y_\Omega \equiv -1$ if $x_1<0.25$. We are able to solve the finite dimensional problem, but up to our best knowledge, there is no proof available about the convergence of the discrete optimal solutions to the continuous optimal solution.

To solve the problem we follow a nested mesh strategy. The number of pcg iterations per Newton iterate grow as $h$ tends to zero, as is to be expected from Corollary \ref{C3.2}; see Table \ref{T7}.

The constraints have been chosen in such way that we seemingly find both a bang-bang part and a singular arc. We have sketched a plot of the discrete optimal control for $h=0.2\times 2^{-8}$ in Figure \ref{F1}. The boundary has been stretched on a 1D line and its corners have been marked with circles on the graph of $\bar u_h$.

\begin{table}[h!]
\tbl{Convergence history for a bilaterally constrained 2D problem with $\nu=0$.}
{
\begin{tabular}{cc|c|c|c}\toprule
&$h$   & Newton & pcg  & $J(\bar u_h)$ \\ \midrule
&$0.2\times2^{-0}$ &  2 &  31 & 0.0694073016\\
&$0.2\times2^{-1}$ &  3 &  61 & 0.0849064798\\
&$0.2\times2^{-2}$ &  4 &  104 & 0.0920895196\\
&$0.2\times2^{-3}$ &  4 &  131 & 0.0971597067\\
&$0.2\times2^{-4}$ &  2 &  87 & 0.0998850162\\
&$0.2\times2^{-5}$ &  4 &  194 & 0.1011428500\\
&$0.2\times2^{-6}$ & 1 &  73 & 0.1017242162\\
&$0.2\times2^{-7}$ & 2 &  185 & 0.1020255084\\
&$0.2\times2^{-8}$ &   2 &  236 & 0.1021724081       \\
\bottomrule
\end{tabular}
}
\label{T7}
\end{table}

\begin{figure}
\begin{center}
\includegraphics[scale = 0.5]{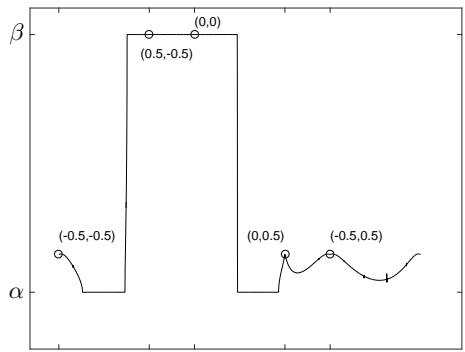}
\caption{\label{F1}Discrete optimal control at the finest level for Example \ref{Enu0}. The polygonal boundary has been stretched over a line and the vertexes have been marked with circles on the graph of $\bar u_h$.}
\end{center}
\end{figure}
\end{example}
\section{State constraints}
\label{S5}
{In the rest of the work, we will suppose that $d=2$,} $\Omega$ is convex and $\Gamma$ is polygonal. According to \cite{Mateos-Neitzel2015}, if problem $(P^S)$ admits a feasible Slater point, then it has a unique solution $\bar u\in H^{1/2}(\Gamma)$ and there exist $\bar y\in H^1(\Omega)\cap C(\bar\omega)$, $\bar\varphi\in W^{1,t}_0(\Omega)$ for all $t<2$ and two nonnegative measures $\bar\mu^+,\ \bar\mu^-\in \mathcal{M}(\bar\omega)$ such that
\begin{subequations}
\begin{align}
-\Delta\bar y = 0\mbox{ in }\Omega,\ \bar y = \bar u\mbox{ on }\Gamma,\label{E4.1a}\\
-\Delta\bar\varphi = \bar y-y_\Omega+\bar\mu^+-\bar\mu^-\mbox{ in }\Omega,\ \bar\varphi = 0\mbox{ on }\Gamma,\label{E4.1b}\\
\bar u(x) = \frac{1}{\nu}\partial_n\bar\varphi(x)\mbox{ on }\Gamma,\label{E4.1c}\\
\langle\bar\mu^+-\bar\mu^-,y-\bar y\rangle \leq 0\ \forall y\in K_{a,b},\label{E4.1d}
\end{align}
\end{subequations}
and $\supp\bar\mu^+\subset\{\bar y = b\},$ $\supp\bar\mu^-\subset\{\bar y = a\}$. In this case the adjoint state equation \eqref{E4.1b} must be understood in the transposition sense, \cite[Equation (4)]{Mateos-Neitzel2015}, and $\langle\cdot,\cdot\rangle$ denotes the duality product between $\mathcal{M}(\bar\omega)$ and $C(\bar\omega)$.

As is pointed out in \cite[page 197]{Bergounioux-Kunisch2002}, or \cite[page 878]{Hintermuller-Ito-Kunisch2003}, the PDAS or semismooth Newton methods described in the previous section are not applicable to this problem, since the multiplier is a measure. In \cite{Ito-Kunisch2003} a Moreau-Yosida regularization is proposed. This strategy is further investigated in \cite{Hintermuller-Kunisch2006}. Let us briefly describe this method. Again, to simplify the notation, we will do the exposition for the unilateral constraint $y\leq b$ in $\bar\omega$.

Given a shift function $\mu^*\in L^q(\omega)$ for some $q>2$ and a parameter $\gamma>0$, we will solve the unconstrained problem
\[(Q^\gamma) \min_{u\in L^2(\Gamma)} J(u) +\frac{1}{2\gamma}\int_{\bar\omega}\max\{0,\mu^*+\gamma(Su-b)\}^2 dx.\]
This problem has a unique solution $u^\gamma\in H^{1/2}(\Gamma)$ for every $\gamma >0$. Since the functional is of class $C^1$, first order optimality conditions can be derived directly from the work \cite{Casas-Raymond2006}: there exist unique $y^\gamma\in H^1(\Omega)$ and $\varphi^\gamma\in H^1_0(\Omega)$ such that
\begin{subequations}
\begin{align}
-\Delta y^\gamma &= 0\mbox{ in }\Omega,\ y^\gamma = u^\gamma\mbox{ on }\Gamma,\label{E4.2a}\\
-\Delta\varphi^\gamma &= y^\gamma-y_\Omega+\chi_{\bar\omega}\max\{0,\mu^*+\gamma(y^\gamma-b)\} \mbox{ in }\Omega,\ \varphi^\gamma = 0\mbox{ on }\Gamma,\label{E4.2b}\\
\nu u^\gamma &= \partial_n\varphi^\gamma\mbox{ on }\Gamma.\label{E4.2c}\end{align}
\end{subequations}

\begin{theorem}\label{T51}The semismooth Newton method to solve \eqref{E4.2a}, \eqref{E4.2b}, \eqref{E4.2c} converges locally $q$-superlinearly.
\end{theorem}
\begin{proof}
We may write the system \eqref{E4.2a}, \eqref{E4.2b}, \eqref{E4.2c} as the equation $G(u)=0$, where $G:L^2(\Gamma)\to L^2(\Gamma)$ is given by
\[G(u) = \nu u+S^*S u-S^*y_\Omega +S^*\big( \chi_{\bar\omega}\max\{0,\mu^*+\gamma(Su-b)\} \big).\]
Using the regularity results in \cite{AMPR2015,Mateos-Neitzel2015}, together with the election of $\mu^*$ in $L^q(\omega)$  we have that $\mu^*+\gamma(Su-b)\in L^q(\Omega)$ for some $q>2$, and hence $G$ is semismooth in the sense stated in \cite{Hinze_Pinnau_Ulbrich_Ulbrich2009}; see Theorems 2.13 and 2.10(c) in the aforementioned reference.

Define now ${\omega_{A}}(u)=\{x\in \bar\omega:\ \mu^*+\gamma(Su-b) > 0\}$. A slant differential of $G(u)$ is given by the expression

\[M(u) v =  \nu v + S^*Sv +\gamma S^*\chi_{\omega_{A}(u)} Sv.\]
Let us see that the inverse of $M(u)$ is uniformly bounded in $\mathcal{L}(L^2(\Gamma),L^2(\Gamma))$ for all $u\in L^2(\Gamma)$. For any $z\in L^2(\Gamma)$, the equation $M(u)v=z$ can be read as
\begin{eqnarray*}
-\Delta y = 0\mbox{ in }\Omega,&& y=v\mbox{ on }\Gamma,\\
-\Delta\varphi = y+\gamma \chi_{\omega_{A}(u)} y\mbox{ in }\Omega,&& \varphi=0\mbox{ on }\Gamma,\\
&&\nu v  = \partial_n\varphi + z\mbox{ on }\Gamma.
\end{eqnarray*}
Using the definition of solution in the transposition sense and the last equation of this system, we have that
\[
(y,y+\gamma\chi_{\omega_{A}(u)} y)_\Omega = -(v,\partial_n\varphi) = -\nu(v,v)_\Gamma+\nu(v,z)_\Gamma
\]
and hence
\[\nu \|v\|_{L^2(\Gamma)}^2 = \nu(v,z)_\Gamma-\int_\Omega y^2 -\gamma\int_{\omega_{A}(u)}y^2\leq \nu\|v\|_{L^2(\Gamma)}\|z\|_{L^2(\Gamma)}.\]
So we have that $\|v\|_{L^2(\Gamma)}\leq \|z\|_{L^2(\Gamma)}$ and therefore $M(u)$ is uniformly bounded in $\mathcal{L}(L^2(\Gamma),L^2(\Gamma))$ for all $u\in L^2(\Gamma)$.
\end{proof}
 When $\gamma\to+\infty$, $u^\gamma\to\bar u$; see \cite[Theorem 3.1]{Ito-Kunisch2003} or \cite[Pr 2.1]{Hintermuller-Kunisch2006}.

\medskip

Let us turn now to the discrete problem $(P^S_h)$. Consider the space $\mathcal{M}_h\subset\mathcal{M}(\bar\omega)$ which is spanned by the Dirac measures corresponding to the nodes $\{x_j\}_{j\in \indice{J}}$, where $\indice{J}=\{j\colon x_j\in\bar\omega\}$. Following \cite[Corollary 2]{Mateos-Neitzel2015}, if the continuous problem $(P^S)$ has a regular feasible Slater point, then $(P^S_h)$ has a unique solution $\bar u_h$. Moreover  there exist unique $\bar y_h\in K_{a,b}^h$, $\bar\varphi_h\in Y_{h,0}$ and two nonnegative measures $\bar\mu_h^+, \bar\mu_h^-\in \mathcal{M}_h$ such that
\begin{subequations}
\begin{align}
a( \bar y_h,z_h)&=0\ \forall z_h\in Y_{h0},\ \bar y_h = \bar u_{h}\mbox{ on }\Gamma,\label{E5.3a}\\
a( z_h,\bar\varphi_h)&=(\bar y_h-y_\Omega,z_h)+
\langle \bar\mu_h^+-\bar\mu_h^-,z_h\rangle\ \forall z_h\in Y_{h0},\\
\langle \bar\mu_h^+-\bar\mu_h^-,y_h-\bar y_h\rangle&\leq 0\ \forall y_h\in K_{a,b}^h,\label{E5.3c}\\
\nu\bar u_h&=\partial_n^h\bar\varphi_h\mbox{ on }\Gamma, \label{E5.3d}
\end{align}
\end{subequations}
and $\supp\bar\mu_h^+\subset\{x_j:\,\bar y(x_j)=b(x_j)\}$, $\supp\bar\mu_h^-\subset\{x_j:\,\bar y(x_j)=a(x_j)\}.$

Since we are dealing with a finite dimensional problem and the $\max$ function is known to be semismooth in finite dimension, we could think about applying directly a semismooth Newton method as described in \cite{Bergounioux-Kunisch2002}. Nevertheless, the other assumption fails here: the slant derivative may not have an inverse.
Let us show this. Consider
\[F(\bm{u},\bm{\mu}) = \left(\begin{array}{c}
\matriz{A}\bm{u}-\bm{f}+\matriz{S}^T\bm{\mu} \\
\bm{\mu} -\max\{ 0,\bm{\mu}+\gamma(\matriz{S}\bm{u}-\bm{b})
\end{array}\right).
\]
It is clear that the optimality system \eqref{E5.3a}--\eqref{E5.3d} is equivalent to the equation $F(\bm{u},\bm{\mu})=\bm{0}$. We define the sets of active and free nodes related to a pair control-multiplier as
\[\indice{A}(u_h,\bm{\mu}) = \{j\in \indice{J}:\mu_j+\gamma(S_h(u_h)(x_j)-b_j) >0\},\ \indice{F} = \indice{J} \setminus \indice{A}.\]
The slant derivative of $F(\bm{u},\bm{\mu})$ is given by
\[M(\bm{u},\bm{\mu}) = \left(\begin{array}{cc}
\matriz{A} & \matriz{S}^T \\
-\gamma \matriz{I}_{\indice{J},\indice{A}}\matriz{I}_{\indice{A},:}\matriz{S} & \matriz{I}_{\indice{J},\indice{F}}\matriz{I}_{\indice{F},\indice{J}}
\end{array}\right).\]
\begin{lemma}The inverse of $M(\bm{u},\bm{\mu})$ may not exist for some $(\bm{u},\bm{\mu})$.
\end{lemma}
\begin{proof}
For any $(\bm{z},\bm{\delta})\in\mathbb{R}^{N_\indice{B}}\times\mathbb{R}^{N_\indice{J}}$, we have that

\[M(\bm{u},\bm{\mu})\left(\begin{array}{c}\bm{v}\\ \bm{\eta}\end{array}\right) =
\left(\begin{array}{c}\bm{z}\\ \bm{\delta}\end{array}\right)
\iff
\left\{\begin{array}{rcl}
\matriz{A}\bm{v} + \matriz{S}^T\bm{\eta} & = & \bm{z}, \\
-\gamma\matriz{I}_{\indice{J},\indice{A}} \matriz{I}_{\indice{A},:} \matriz{S} \bm{v} +
\matriz{I}_{\indice{J},\indice{F}} \matriz{I}_{\indice{F},\indice{J}}\eta & = & \bm{\delta}.
\end{array}\right.
\]
From the second equation we have that
\[\gamma \matriz{I}_{\indice{A},:} \matriz{S}\bm{v} = - \matriz{I}_{\indice{A},\indice{J}}\bm{\delta}.\]
If $N_\indice{A} > N_\indice{B}$ we have a linear system with more equations than variables, and hence it will not be consistent for at least one value of $\bm{\delta}$ and $M$ will not have an inverse.
\end{proof}

Instead, we use a Moreau-Yosida penalization of $(P^S_h)$. Again we will write only the case of unilateral upper constraint. For some shift function $\mu^*\in Y_h$ such that $\mu^*(x_j)=0$ if $j\not\in\indice{J}$, and a parameter $\gamma>0$, a direct discretization of $(Q^\gamma)$ could be
\[(\tilde Q^\gamma_h)\min_{u_h\in U_h}J_h(u_h)+\frac{1}{2\gamma}\int_{\bar\omega}
\max\{0,\mu^*+\gamma(S_hu_h-b)\}^2dx.\]
Problem $(\tilde Q^\gamma_h)$ has a unique solution $u^\gamma_h\in U_h$ that converges to $u^\gamma$ as $h\to 0$. Since the functional is not of class $C^2$, the problem does not fit exactly into the framework of \cite{Casas-Raymond2006,Deckelnick-Gunther-Hinze2009,Mateos-Neitzel2015}, so we will give a sketch of the proof of this fact.
\begin{lemma}\label{L63}Problem $(\tilde Q^\gamma_h)$ has a unique solution $u^\gamma_h\in U_h$. Moreover, there exist unique $y_h^\gamma\in Y_h$ and $\varphi_h^\gamma\in Y_{h0}$ such that
\begin{subequations}
\begin{align}
a( y^\gamma_h,z_h)&=0\ \forall z_h\in Y_{h0},\  y^\gamma_h =  u^\gamma_{h}\mbox{ on }\Gamma,\label{E5.4A}\\
a(z_h,\varphi^\gamma_h)&=( y^\gamma_h-y_\Omega+
\max\{0,\mu^*+\gamma(y_h^\gamma-b)\},z_h)_\Omega
\ \forall z_h\in Y_{h0},\label{E5.4B}\\
 \nu u^\gamma_h&=\partial_n^h\varphi^\gamma_h\mbox{ on }\Gamma.\label{E5.4C}
\end{align}
\end{subequations}
Finally, $u_h^\gamma\to u^\gamma$ in $L^2(\Gamma)$ and there exists some $\lambda>0$ such that
\[\|u^\gamma-u^\gamma_h\|_{L^2(\Gamma)}\leq C h^\lambda.\]
\end{lemma}
\begin{proof}Existence and uniqueness of solution follows immediately from the coerciveness of the discrete penalized functional. Since this discrete functional is also of class $C^1$, first order optimality conditions are standard.

Consider two controls $u_1,u_2\in L^2(\Gamma)$ and its related states $y_i$ and adjoint states $\varphi_i$ w.r.t. problem $(Q^\gamma)$. Using the definition of solution in the transposition sense for the states and both the state and adjoint state equations \eqref{E4.2a} and \eqref{E4.2b} we have (shorting $z_i = \mu^*+\gamma(y_i-b)$, $i=1,2$)
\begin{eqnarray*}
(-\partial_n\varphi_1 + \partial_n\varphi_2,u_1-u_2)_\Gamma  &=&
(y_1-y_2,-\Delta\varphi_1+\Delta\varphi_2)_\Omega \\
 & = & (y_1-y_2, y_1+\chi_\omega\max\{0,z_1\}-y_2-\chi_\omega\max\{0,z_2\})_\Omega\\
 &=&\|y_1-y_2\|^2_{L^2(\Omega)}+\frac{1}{\gamma}(z_1 - z_2,\max\{0,z_1\}-\max\{0,z_2\})_\omega\\
 & \geq& 0.
\end{eqnarray*}
The rest is quite standard:
\begin{eqnarray*}
\nu\|u^\gamma- u^\gamma_h\|_{L^2(\refre{\Gamma})}^2&\leq&
(-\partial_n\varphi^\gamma+\nu u^\gamma+\partial_n\varphi_{u_h^\gamma}- \nu u_h^\gamma, u^\gamma-u_h^\gamma)_\Gamma\\
& = & (-\partial_n\varphi^\gamma+\partial_n^h\varphi_h^\gamma+\nu u^\gamma-\nu u_h^\gamma, u^\gamma-  u_h^\gamma)_\Gamma \\ && +(-\partial_n^h\varphi_h^\gamma+\partial_n\varphi_{u_h^\gamma},u^\gamma- u_h^\gamma)_\Gamma.
\end{eqnarray*}
For fixed $\gamma$, $u_h^\gamma$ is uniformly bounded in $L^2(\Gamma)$ because $\| u_h^\gamma\|_{L^2(\Gamma)}^2\leq 1/\nu\|y_\Omega\|_{L^2(\Omega)}^2 + 1/\gamma/\nu\|\max\{0,\mu^*-\gamma b\}\|_{L^2(\Omega)}^2$, and therefore using \cite[Theorem 5.7]{Casas-Raymond2006} we have that $\|-\partial_n^h\varphi_h^\gamma+\partial_n\varphi_{u_h^\gamma}\|\leq C h^\lambda$ for some $\lambda>0$. For the first addend we obtain directly $0$ using
first order optimality conditions \eqref{E4.2c} and \eqref{E5.4C}.
\end{proof}

A practical way of computing the integral in the penalty term is to use the lumped mass matrix $\matriz{L}\in\mathbb{R}^{N\times N}$. Since we are going to use it only for integrals in $\bar\omega$, we define it as
\[\matriz{L}_{i,j} = 0\mbox{ if }i\neq j,\ \matriz{L}_{j,j} = 0\mbox{ if }j\not\in\indice{J},\ \matriz{L}_{j,j}= \sum_{k=1}^N \matriz{M}_{k,j}.\mbox{ if }j\in\indice{J}.\]
Therefore, we will be solving
\[(Q^\gamma_h)\min_{u_h\in U_h}J_h(u_h)+\frac{1}{2\gamma}\sum_{j\in\indice{J}}\matriz{L}_{j,j}
\max\{0,\mu^*(x_j)+\gamma(S_hu_h(x_j)-b(x_j))\}^2.\]
 First order optimality conditions for this problem read like
\begin{subequations}
\begin{align}
a( y^\gamma_h,z_h)&=0\ \forall z_h\in Y_{h0},\  y^\gamma_h =  u^\gamma_{h}\mbox{ on }\Gamma,\label{E5.4a}\\
a( z_h,\varphi^\gamma_h)&=( y^\gamma_h-y_\Omega,z_h)_\Omega\notag\\
&\quad +\sum_{j\in\indice{J}}\matriz{L}_{j,j}
\max\{0,\mu^*(x_j)+\gamma(y_h^\gamma(x_j)-b(x_j))\}z_h(x_j)
\ \forall z_h\in Y_{h0},\label{E5.4b}\\
 \nu u^\gamma_h&=\partial_n^h\varphi^\gamma_h\mbox{ on }\Gamma.\label{E5.4c}
\end{align}
\end{subequations}
This nonlinear system can be solved using a semismooth Newton method.
To this end, we will use the nonlinear operator $G_h:\mathbb{R}^{N_\indice{B}}\to \mathbb{R}^{N_\indice{B}}$ defined by
\[G_h(\bm{u}) =
\matriz{A}\bm{u}-\bm{f}+
\matriz{S}^T\matriz{L} \max\{\bm{0},\bm{\mu}^*+\gamma(\matriz{S}\bm{u}-\bm{b})\}.\]
Thus, solving \eqref{E5.4a}--\eqref{E5.4c} is equivalent to solving $G_h(\bm{u})=\bm{0}$.

For every $u_h\in U_h$
we define sets of active and free nodes as
\begin{equation}\label{E5.5}
\indice{A}_\omega(u_h,\gamma,h) = \{j\in \indice{J}: \mu^*_j+\gamma(S_hu_h(x_j)-b_j) > 0\},\ \indice{F}_\omega(u_h,\gamma,h) =\indice{I}\setminus \indice{A}_\omega(u_h,\gamma,h).\end{equation}
Abusing notation we will drop some or all of the arguments or will use the vector notation when this does not lead to confusion, e.g. $\indice{A}_\omega$, $\indice{A}_\omega(\bm{u})$.
Notice that if $\mu^*\equiv 0$, then $\indice{A}_\omega(u_h,\gamma,h)$ is independent of $\gamma$.
For the sake of notation, it is also convenient to define for every set $\indice{A}_\omega\subseteq \indice{J}$ the diagonal matrix $\matriz{H}(\indice{A}_\omega)\in\mathbb{R}^{N\times N}$ such that
\[\matriz{H}_{i,j}=\delta_{i,j}\left\{\begin{array}{cc}
0 & \mbox{ if }j\not\in \indice{A}_\omega\\
\matriz{L}_{j,j} & \mbox{ if }j\in \indice{A}_\omega.
\end{array}\right.
\]
Abusing notation, we will often write $\matriz{H}(\bm{u})=\matriz{H}(\indice{A}_\omega(\bm{u}))$ or even we will write $\matriz{H}$ when this does not lead to confusion. The proof of the following result is as the corresponding one in infinite dimension.
\begin{theorem}\label{T64}$G_h(\bm{u})$ is slantly differentiable, a slant differential is given by
\[M(\bm{u})\bm{v} = \matriz{A}\bm{v}+\gamma\matriz{S}^T\matriz{H}(\bm{u}) \matriz{S}\bm{v}\]
and it has a uniformly bounded inverse w.r.t. $\bm{u}$.
\end{theorem}

Notice that using this $\matriz{H}$ notation, we can write
\[G_h(\bm{u}) =
\matriz{A}\bm{u}-\bm{f}+
\matriz{S}^T\matriz{H}(\bm{u})\big(\bm{\mu}^*+\gamma(\matriz{S}\bm{u}-\bm{b})\big),\]
and therefore, for a given $\bm{u}$, and denoting $\matriz{H}=\matriz{H}(\bm{u})$, one Newton step reads like
\begin{equation}\label{E2500}\matriz{A}\bm{u}^+ +\gamma\matriz{S}^T\matriz{H} \matriz{S}\bm{u}^+ = \bm{f}+
\matriz{S}^T\matriz{H}(\gamma\bm{b}-\bm{\mu}^*).
\end{equation}
Let us comment that for the computation of $\bm{w} = \matriz{S}^T\matriz{H}\bm{y}$ for some $\bm{y}\in \mathbb{R}^N$, first we solve
\begin{equation}\label{E2300a}\matriz{K}_{\indice{I},\indice{I}}\bm{\phi}_\mathrm{s} = \matriz{H}_{\indice{I},:}\bm{y},\end{equation}
and, next, we have
\[\bm{w} = \matriz{H}_{\indice{B},:}\bm{y}-\matriz{K}_{\indice{B},\indice{I}}\bm{\phi}_\mathrm{s} = -\matriz{K}_{\indice{B},\indice{I}}\bm{\phi}_\mathrm{s}\]
because $\matriz{H}$ is diagonal and its nonzero components  correspond to nodes that lie in $\bar\omega$, and are hence interior to $\Omega$.

Again a preconditioned conjugate gradient method can be used to solve this system, provided an efficient way of computing $\bm{d}=(\matriz{A} + \gamma\matriz{S}^T\matriz{H} \matriz{S}) \bm{v}$ and the second member of the system. Notice that in each of the algorithms \ref{A400b} and \ref{A500b2} the computation of $\bm{\phi}$ can be done with just one system solve.

\begin{algorithm}[H]
\caption{Function $\bm{d}=(\matriz{A} + \gamma\matriz{S}^T\matriz{H} \matriz{S}) \bm{v}$\label{A400b}}
\DontPrintSemicolon
solve \eqref{E200} for $\bm{y}$\;
solve \eqref{E300} for $\bm{\phi}_\mathrm{r}$\;
solve \eqref{E2300a} for $\bm{\phi}_\mathrm{s}$\;
set $\bm{\phi}=\bm{\phi}_\mathrm{r} + \gamma \bm{\phi}_\mathrm{s}$\;
set $\bm{d} = \matriz{M}_{\indice{B},:}\bm{y}  -
\matriz{K}_{\indice{B},\indice{I}}\bm\phi+\nu\matriz{B}_{\indice{B},\indice{B}}\bm{u}$\;
\end{algorithm}

\begin{algorithm}[H]
\caption{Computation of $\bm{c} = \bm{f}+
\matriz{S}^T\matriz{H}(\bm{u})(\gamma\bm{b}-\bm{\mu}^*)$\label{A500b2}}
\DontPrintSemicolon
set $\bm{y}=\bm{y}_\Omega$;
solve \eqref{E300} for $\bm{\phi}_\mathrm{r}$\;
set $\bm{y} = \gamma\bm{b}-\bm{\mu}^*$;
solve \eqref{E2300a} for $\bm{\phi}_\mathrm{s}$\;
set $\bm{\phi}=\bm{\phi}_\mathrm{r} + \bm{\phi}_\mathrm{s}$\;
set $\bm{c} = \matriz{M}_{\indice{B},:}\bm{y}_\Omega-
\matriz{K}_{\indice{B},\indice{I}}\bm\phi$\;
\end{algorithm}

Alternatively, the solution to \eqref{E2500} can be obtained solving

\begin{equation}
\begin{array}{rcl}
\left[\begin{array}{cc}
\matriz{M}+\nu\matriz{B}+\gamma\matriz{H} & -\matriz{K}_{:,\indice{I}}  \\
-\matriz{K}_{\indice{I},:} & \matriz{O}_{\indice{I},\indice{I}}
\end{array}\right]
\left[\begin{array}{c}
\bm{y}^+ \\
\bm{\varphi}_\indice{I}^+ \end{array}
\right]
&=&
\left[\begin{array}{c}
\matriz{M}\bm{y}_\Omega+\matriz{H}(\gamma\bm{b}-\bm{\mu}^*))\\
\bm{0}
\end{array}
\right]\\
\bm{u}^+&=&\bm{y}^+_\indice{B}.
\end{array}
\label{E2200}
\end{equation}


Finally, the semismooth Newton algorithm to solve our problem reads as:

\LinesNumbered
\begin{algorithm}[H]
\caption{Semismooth Newton method for $(Q^\gamma_h)$\label{A600}}
\DontPrintSemicolon
Set $k=0$ and initialize $\bm{u}_0$; compute $\indice{A}_{\omega,0} = \indice{A}_\omega(\bm{u}_{0})$\;
Set $\indice{A}_\omega = \indice{A}_{\omega,k}$ and $\matriz{H} = \matriz{H}(\indice{A}_\omega )$\; \label{L2b}
Compute $\bm{u}_{k+1} = \bm{u}^+$  using \eqref{E2500} [or \eqref{E2200}]\;\label{A600L3}
Set $\indice{A}_{\omega,k+1} = \indice{A}_\omega(\bm{u}_{k+1})$\;
Stop or set $k= k+1$ and return to \ref{L2b}\;
\end{algorithm}

\begin{remark}As we noticed in Remark \ref{R45},  the only information from iteration $k$ used to compute iteration $k+1$ is the set of active indexes. Therefore, when $\indice{A}_\omega(\bm{u}_{k+1})=\indice{A}_\omega(\bm{u}_{k})$ we have reached an stationary point. This is usually the criterion used to stop the semismooth Newton method.

Nevertheless, as we will see below
 in the context of algorithms \ref{A700} and \ref{A1000}, the main use for the solution of $(Q_h^{\gamma})$ will be to provide an initial guess for the next step in those procedures, so it does not seem necessary to solve exactly $(Q_h^{\gamma})$ in general. Following \cite{Hintermuller-Kunisch2006}, we may implement the following stopping criterion for Algorithm \ref{A600}. After step \ref{A600L3}, we compute an approximation of the multiplier $\bm{\mu}_{k+1}\in \mathbb{R}^N$
as
\[\mu_{k+1,j} = 0\mbox{ if }j\not\in \indice{A}_{\omega,k},\ \mu_{k+1,j} = \mu^*_j+\gamma(y_{k+1,j}-b_j)\mbox{ if }j\in \indice{A}_{\omega,k}\]
and we may stop the algorithm if
\[ \|\mu_{k+1}-\max(\mu^*+\gamma(y_{k+1}-b),0\|_{L^2(\Omega)}< \varepsilon_\lambda.\]
In practice, this quantity gives a good measure of the change in the active set between iterates as well as the unfeasibility combined with the penalization parameter $\gamma$.
\end{remark}

We state the convergence result for the Newton method. It follows directly from Theorem \ref{T64}
\begin{theorem}The semismooth Newton method to solve $(Q^\gamma_h)$ described in Algorithm \ref{A600} converges locally q-superlinearly.\end{theorem}

In all the examples below we have taken the shift $\mu^*=0$.

\begin{example}\label{Ex5.6}In general it is not a good idea to solve directly $(Q_h^\gamma)$ for some  $\gamma$ big enough. The resulting intermediate problems are usually very ill conditioned.

We resume the 2D problem taken from \cite{Mateos-Neitzel2015} and described in Example \ref{Ex31}. As in \cite{Mateos-Neitzel2015}, we take $\omega$ the ball centered at $(-0.1,-0.1)$ with radius $0.2$ and $b\equiv 0.15$. We use Algorithm \ref{A600} taking $\bm{u}_0=\bm{0}$. The results are summarized in Table \ref{T5.5}. We measure the unfeasibility of the state as the maximum constraint violation $\mathrm{mcv}(y_h) = \|\max(y_h-b,0)\|_{L^\infty(\bar\omega)}$.
\begin{table}[h!]
\tbl{Direct application of Algorithm \ref{A600} for $\gamma=10^{9}$}
{
\begin{tabular}{cr|c|c|c|c|c}\toprule
&$h$&$\gamma$ & $\mathrm{mcv}\,y_h^\gamma$ & Newton&pcg & $J_h(u_h^\gamma)$ \\ \midrule
&$0.2\times 2^{-4}$ & 1E+09 & 6.0E-07 & 19 &  470 & 0.3552476482\\
&$0.2\times 2^{-5}$ & 1E+09 & 1.1E-06 & 19 &  468 & 0.3552689635\\
&$0.2\times 2^{-6}$ & 1E+09 & 1.6E-06 & 21 &  467 & 0.3552757641\\
&$0.2\times 2^{-7}$ & 1E+09 & 2.5E-06 & 21 &  459 & 0.3552778183\\
&$0.2\times 2^{-8}$ & 1E+09 & 2.7E-06 & 21 &  451 & 0.3552784696\\
\bottomrule
\end{tabular}
}
\label{T5.5}
\end{table}

The first thing that can be observed is that $\gamma$ cannot be too big w.r.t. the problem size. Although the algorithm should converge in finite time, a value too big for $\gamma$ will make the active sets of the intermediate steps fluctuate, and in practice it may not stop. This is what happens in this case with meshes coarser than the first one exposed in the table. On the other hand, it can also be noticed that the amount of Newton iterations looks mesh independent for $h$ small enough. This was to be expected because the infinite dimensional version of the method is convergent (cf. Theorem \ref{T51}).

Nevertheless, the computational effort can be better measured by the total number of conjugate gradient iterations made when we solve \eqref{E2500} in step \ref{A600L3} of Algorithm \ref{A600}. Solving the solution in the finest mesh takes 8246 seconds.

\end{example}

In \cite{Hintermuller-Kunisch2006} a continuation strategy is proposed. This reduces considerably the computational effort. We will use subscripts for the iteration number of the semismooth Newton method described in Algorithm \ref{A600} and superscripts for the numbering of the iterates of the continuation strategy provided in Algorithm \ref{A700}. For some sequence $\{\tau^{n}\}_{n\geq 1}$ such that $\tau^n>1$ for all $n\geq 1$, we have

\LinesNumbered
\begin{algorithm}[H]
\caption{Continuation strategy to solve $(Q^\gamma_h)$\label{A700}}
\DontPrintSemicolon
Set $n=0$ and initialize $\gamma^0$ and $\bm{u}^0$\; 
Compute $\bm{u}^{n+1}$, the solution of $(Q^{\gamma^n}_h)$ using Algorithm \ref{A600} with seed $\bm{u}^n$\;\label{A700L2}
Set $\gamma^{n+1} = \tau^{n+1}\gamma^n$\;
Stop or set $n = n+1$ and return to \ref{A700L2}\;
\end{algorithm}

There are still three important details to be explained about Algorithm \ref{A700}: first and most important, the choice of the initial guess $\bm{u}^0$; next, the appropriate values for $\tau^n$; and finally, a suitable stopping criterion.

To obtain a good initial guess, we will solve the problem in a coarser mesh with a smaller value of $\gamma$. Since we are going to deal with meshes of different sizes, we will use the notation $\indice{A}_\omega({u}_h,\gamma,h)$ when needed. We also recall that $I_h$ is the pointwise interpolation operator from $C(\bar\Gamma)$ onto $U_h$. Given $(h,\gamma^0)$, we fix $M>0$ and pick a finite sequence $(h_{n},\gamma_{n})_{n=0}^{M}$ such that $h_{n}$ is decreasing, $h_M = h$, $\gamma_{n}$ is increasing and $\gamma_{M} \leq \gamma^0$.

\begin{algorithm}[H]
\caption{Nested concept to choose $\bm{u}^0$ in Algorithm \ref{A700}\label{A1000}}
\DontPrintSemicolon
Set $n = 0$ and initialize  ${u}_{h_{0}}\in U_{h_{0}}$\;
Solve $(Q_{h_{n}}^{\gamma_{n}})$ using Algorithm \ref{A600} with initial guess ${u}_{h_n}$\;\label{L2A1000}
Set $ u_{h_{n+1}} = I_{h_{n+1}} u_{h_{n}}^{\gamma_{n}}\in U_{h_{n+1}}$\;\label{L3A1000}
Set $n=n+1$\;
If $n < M$ goto \ref{L2A1000}\;
Return $u_{h_{M}}\in U_h$
\end{algorithm}
Although Algorithm \ref{A1000} is meaningful for non-nested meshes, the computational effort needed to make the interpolation in step \ref{L3A1000} in non-nested meshes can be considerable, specially for 3D problems.

In \cite{Hintermuller-Kunisch2006} some criteria are given to choose $\tau^{n+1}>1$. In practice, if $\tau^{n+1}$  is very small the algorithm would not advance; if it is very big, we would lose the advantage given by the continuation strategy.

In \cite{Hintermuller-Kunisch2006} the authors stop the continuation strategy if residuals related to the state, the adjoint state and the multiplier are smaller than a certain tolerance of order $O(h^2)$. In our case, since the problem is linear quadratic, the residuals related to the state and the adjoint state are zero (at least up to roundoff error). The residual related to the multiplier can be computed as
\[r_d = \sum_{j\in \indice{A}_\omega^{n+1}} \matriz{L}_{j,j} (y^{n+1}_j-b_j).\]
As an alternative, a tolerance for the maximum constraint violation $e_\infty$ can be used stop if $\mathrm{mcv}(y_h^{\gamma_n})\leq e_{\infty}$.

With all these considerations, we propose the following algorithm. Fix a mesh sequence such that $\{h_j\}_{j=0}^M$ is decreasing, $\gamma_0>0$, $C>0$, $u_{h_0}\in U_{h_0}$, $n_{\max}\in\mathbb{N}$ and a sequence $\{\tau^{n}\}_{n=1}^{n_{\max}+1}$ such that $\tau^n>1$ for all $n$.

\begin{algorithm}[H]
\caption{Solving $(P_h^S)$ \label{A1100}}
\DontPrintSemicolon
Set $j = 0$, $h=h_j$, $u_{h,0} = u_{h_0}$\;
Set $n = 0$\;
Compute $u_h^{\gamma_n}$, the solution of $(Q_{h}^{\gamma_n})$, using Algorithm \ref{A600} with seed $u_{h,n}$\;\label{L3A1100}
\eIf{($r_d<C h^2$ {\bf or} $\mathrm{mcv}(y_h^{\gamma_n})\leq e_{\infty}$) {\bf and} $j<M$}
{
Set $j = j + 1$ and $h = h_{j}$ \tcc*[f]{Refine mesh}\;
Set $u_{h,n+1} = I_{h} u_{h_{j-1}}^{\gamma_{n}}\in U_{h}$ \tcc*[f]{Write last output in the new mesh basis}\;
}
{Set $ u_{h,n+1} = u_h^{\gamma_n}$}
Set $\gamma_{n+1} = \tau^{n+1}\gamma_n$\tcc*[f]{Update $\gamma$}\;
Set $n=n+1$\;
\eIf{( $(r_d<C h^2$ {\bf or} $\mathrm{mcv}(y_h^{\gamma_n})\leq e_{\infty}$) {\bf and} $j==M)$ {\bf or} $n > n_{\max}$}{Stop}{Goto \ref{L3A1100}}
\end{algorithm}

\begin{example}\label{Ex5.7}Now we apply Algorithm \ref{A1100} for $\gamma_0=1$, $\tau^n=10$ and a family of 9 nested meshes of sizes $h_j=0.2*2^{-j}$, $j=0:8$, the coarsest for $h_0$ with 52 nodes and the finest for $h_8$ with  2689537 nodes.

A summary of the results can be found in Table \ref{T5.6}. The total computation time was 1570 seconds, which is a significant improvement compared with the 8246 seconds used by Algorithm \ref{A600} (see Example \ref{Ex5.6}).

\begin{table}[h!]
\tbl{Data for continuation strategy and nested choice of the initial point.}
{
\begin{tabular}{cr|c|c|c|c|l}\toprule
&$j$ &  $\gamma_n$ & $r_d$ & Newton & pcg & $J_h(u_h^\gamma)$ \\ \midrule
& 0 &   1E+00 & 1.20E-02 & 2 &  12 & 0.3434276446\\
& 1 &    1E+01 & 7.24E-03 & 1 &  7 & 0.3464821594\\
& 1 &    1E+02 & 1.96E-03 & 1 &  7 & 0.3512477862\\
& 2 &     1E+03 & 2.35E-04 & 3 &  25 & 0.3544114381\\
& 3 &    1E+04 & 2.40E-05 & 4 &  37 & 0.3550780145\\
& 4 &    1E+05 & 2.39E-06 & 4 &  44 & 0.3552160750\\
& 5 &    1E+06 & 2.38E-07 & 5 &  59 & 0.3552595829\\
&6 & 1E+07 & 2.37E-08 & 5 &  63 & 0.3552727851\\
& 7 &   1E+08 & 2.37E-09 & 5 &  54 & 0.3552770205\\
& 8 & 1E+09 & 2.37E-10 & 5 &  50 & 0.3552784696\\
\bottomrule
\end{tabular}
}
\label{T5.6}
\end{table}
\end{example}
\textbf{Determination of (the support of) the Lagrange multipliers of $(P^S)$.} Comparison of the adjoint state equations for $(P_h^S)$ and $(Q_h^\gamma)$ leads to the approximation formula
\[\bar\mu_h^+ \approx \sum_{j\in\indice{A}_\omega(u_h^\gamma)}\matriz{L}_{j,j} (\mu^*_j +\gamma(y^\gamma_h(x_j)-b_j))\delta_{x_j}.\]

Since $\bar\omega\subset\Omega$, it is very common that the Lagrange multipliers of the original problem are finite sums of Dirac measures centered at points on the boundary of $\omega$, say $\bar\mu^+ = \sum_{k=1}^n \mu_k\delta_{X_k}$ for some $n>0$ and $X_k\in\partial\omega$.

Increasing $\gamma$ will have as a result a better approximation of the Lagrange multiplier and its support. Notice, nevertheless, this will done with great effort and the approximation of both the control, the state and the functional optimal value will not improve in a significant way.
\begin{example}
In Example 5.7 the number of active nodes for the last mesh is 748. All of them are on $\partial\omega$ but they are not isolated. We can see in Table \ref{T5.8} how the number of active nodes decreases as $\gamma$ increases, with little variation of the functional or the state.

\begin{table}[h!]
\tbl{Determination of the active set. Data for $h=0.2\times 2^{-8}$.}
{
\begin{tabular}{cc|r|c|c}\toprule
&$\gamma$ & $N_{\indice{A}_\omega}$ & $\mathrm{mcv}(y_h^{\gamma})$ & $J(u_h^\gamma)$ \\ \midrule
&1E+09 & 748 & 2.7E-06 & 0.3552784696 \\
&1E+10 & 230 & 9.1E-07 & 0.3552787761\\
&1E+11 & 106 & 2.0E-07 & 0.3552788799\\
&1E+12 & 46  & 4.1E-08 & 0.3552789024\\
&1E+13 & 21  & 7.8E-09 & 0.3552789074\\
&1E+14 & 11  & 1.9E-09 & 0.3552789083\\
&1E+15 & 5   & 2.0E-10 & 0.3552789086\\
\bottomrule
\end{tabular}
}
\label{T5.8}
\end{table}
\end{example}
\section{Control and state constraints}
\label{S6}
Again according to \cite{Mateos-Neitzel2015}, if problem $(P^{CS})$ admits a feasible Slater point, then it has a unique solution $\bar u\in H^{1/2}(\Gamma)$ and there exist $\bar y\in H^1(\Omega)\cap C(\bar\omega)$, $\bar\varphi\in W^{1,t}_0(\Omega)$ for all $t<2$ and two nonnegative measures $\bar\mu^+,\ \bar\mu^-\in \mathcal{M}(\bar\omega)$ such that
\begin{subequations}
\begin{align}
-\Delta\bar y = 0\mbox{ in }\Omega,\ \bar y = \bar u\mbox{ on }\Gamma,\label{E6.1a}\\
-\Delta\bar\varphi = \bar y-y_\Omega+\bar\mu^+-\bar\mu^-\mbox{ in }\Omega,\ \bar\varphi = 0\mbox{ on }\Gamma,\label{E6.1b}\\
\bar u(x) = \min\{\beta(x),\max\{\alpha(x),\frac{1}{\nu}\partial_n\bar\varphi(x)\}\}\mbox{ on }\Gamma,\label{E6.1c}\\
\langle\bar\mu^+-\bar\mu^-,y-\bar y\rangle  \leq 0\ \forall y\in K_{a,b}\label{xxxE48d}
\end{align}
\end{subequations}
and $\supp\bar\mu^+\subset\{\bar y = b\},$ $\supp\bar\mu^-\subset\{\bar y = a\}$. As we said in the previous section, a semismooth Newton strategy for this problem is meaningless, so instead we are going to deal with a Moreau-Yosida approximation.
{As we did in the previous sections, we will consider only unilateral constraints $u\leq \beta$ on $\Gamma$ and $y\leq b$ in $\bar\omega$ to simplify the notation.}
For a shift function $\mu^*\in L^q(\omega)$ for some $q>2$ and a parameter $\gamma>0$, we consider the problem
\[(Q^{C,\refre{\gamma}}) \min_{u\in U_{-\infty,\beta}}J(u)+\frac{1}{2\gamma}\int_{\bar\omega}
\max\{0,\mu^*+\gamma(Su-b)\}^2dx.\]
This problem has a unique solution $u^\gamma\in H^{1/2}(\Gamma)$. Moreover, there exist $y^\gamma\in H^1(\Omega)$, $\varphi^\gamma\in H^s(\Omega)$, $s>3/2$, such that
\begin{subequations}
\begin{align}
-\Delta y^\gamma =& 0\mbox{ in }\Omega,\  y^\gamma =  u^\gamma\mbox{ on }\Gamma,\label{xxxE49a}\\
-\Delta\varphi^\gamma =& y^\gamma-y_\Omega+\max\{0,\mu^*+\gamma(y^\gamma-b)\}\mbox{ in }\Omega,\ \varphi^\gamma = 0\mbox{ on }\Gamma,\label{xxxE49b}\\
 (-\partial_n\varphi^\gamma+\nu u^\gamma,u-u^\gamma)& \geq 0\mbox{ for all }u\in U_{-\infty,\beta}.\label{xxxE49c}
\end{align}
\end{subequations}
Define
\[G^\gamma(u) = \nu u-\min\big\{\nu\beta,-S^*Su+S^*y_\Omega-S^*\chi_{\bar\omega}\max\{0,\mu^*+\gamma(Su-b)\}\big\}.\]
It is clear that $u^\gamma$ is the unique solution of $(Q^{c,\gamma})$ if and only if $G^\gamma(u^\gamma)=0$.
For some fixed shift function $\mu^*\in L^q(\omega)$, $q>2$, consider the active sets
\[\Gamma_{A}(u,\gamma) = \{x\in\Gamma:\ -S^*Su+S^*y_\Omega-S^*\chi_\omega\max\{0,\mu^*+\gamma(Su-b)>\nu\beta\}\]
and
\[\omega_{A}(u,\gamma) = \{x\in\bar\omega:\ \mu^*+\gamma(Su-b)>0\}.\]
A slant differential of $G^\gamma(u)$ is given by
\[M^\gamma(u)v = \nu v+\chi_{\Gamma_{A}(u,\gamma)}S^*(1+ \gamma \chi_{\omega_{A}(u,\gamma)})Sv.\]

\begin{theorem}
$G^\gamma$ is slantly differentiable, $M^\gamma(u)$ is a slant differential of $G^\gamma$ and for every fixed $\nu>0$, $M^\gamma(u)$ has an inverse in $\mathcal{L}(L^2(\Gamma),L^2(\Gamma))$  uniformly bounded for all $u\in L^2(\Gamma)$.
The semismooth Newton method $M^\gamma(u^+-u) = -G^\gamma(u)$ converges q-superlinearly.
\end{theorem}
\begin{proof}The proof follows the same lines as those of Lemma \ref{L4.1} and Theorems \ref{T4.2} and \ref{T51}\end{proof}

As we did for the pure control-constrained case, to deal with a problem better suited to the finite dimensional case, we write the optimality condition \eqref{xxxE49c} with the help of a Lagrange multiplier.
There exists $\lambda^\gamma\in L^2(\Gamma)$ such that, for any $c>0$,
\begin{subequations}
\begin{align}
  \nu u^\gamma = &\partial_n\varphi^\gamma-\lambda^\gamma\mbox{ on }\Gamma,\label{xxxE50a}\\
\lambda^\gamma=&\max\{0,\lambda^\gamma+c(u^\gamma-\beta)\}\mbox{ on }\Gamma.\label{xxxE50b}
\end{align}
\end{subequations}
Define now
\[F_c^\gamma(u,\lambda) = \left(
\begin{array}{c}
\nu u + S^*Su-S^*y_\Omega + S^*\chi_{\bar\omega}\max\{0,\mu^*+\gamma(Su-b)\}+\lambda\\
\lambda - \max\{0,\lambda + cu -c\beta\}
\end{array}\right)
\]
We have that
\[\left(\begin{array}{cc}
\nu I+S^*(1+ \gamma \chi_{\omega_{A}(u,\gamma)})S & I\\
c\chi_{\Gamma_{\indice{A}(u,\lambda)}}I & \chi_{\Gamma_{\indice{F}(u,\lambda)}}I
\end{array}\right)\in\partial^{CL}F_c^\gamma(u,\lambda)\]
and hence a Newton-like method to solve $F_c^\gamma(u,\lambda)=0$  is given by
\begin{eqnarray*}
\nu u^+ &=& -S^*(1+ \gamma \chi_{\omega_{\indice{A}(u,\gamma)}})Su^+-y_\Omega) -\lambda^+,\\
u^+ &=& \beta\mbox{ on }\Gamma_{\indice{A}(u,\lambda)},  \\
\lambda^+& = &0\mbox{ on }\Gamma_{\indice{F}(u,\lambda)}.
\end{eqnarray*}
Although $F_c^\gamma$ is known not to be slantly differentiable, it can be proved as in Theorem \ref{T43} that the sequence generated by the above described Newton-like method to solve $F_c^\gamma(u,\lambda)=0$ is the same than the one generated by the semismooth Newton method to solve $G^\gamma(u)=0$ provided we take the same initial guess $u_0$, $c=\nu$ and $\lambda_0=S^*(y_\Omega-(1+ \gamma \chi_{\omega_{\indice{A}(u_0,\gamma)}})Su_0)-\nu u_0$.

Let us turn now to the finite dimensional problem.
We can write the approximation of problem $(Q^{C,\refre{\gamma}}_h)$ as a constrained optimization problem in $\mathbb{R}^{N_\indice{B},1}$:
\[
(Q^{C,\refre{\gamma}}_h)\left\{\begin{array}{l}\min\frac{1}{2}\bm{u}^T\matriz{A}\bm{u}-\bm{f}^T\bm{u}+
\frac{1}{2\gamma}\max\{\bm{0},\bm{\mu}^*+\gamma(\matriz{S}\bm{u}-\bm{b})\}^T\matriz{L}\max\{\bm{0},\bm{\mu}^*+\gamma(\matriz{S}\bm{u}-\bm{b})\}\\
\mbox{subject to }
\bm{u}\leq\bm{\beta}.
\end{array}
\right.\]
Existence and uniqueness of solution of this problem, as well as error estimates for the difference $\|u^\gamma-u^\gamma_h\|_{L^2(\Gamma)}$ can be proved as we did for $(\tilde Q^\gamma)$ in Lemma \ref{L63}. Since we have control constraints, instead of the point-wise interpolation, the Casas-Raymond interpolate \cite[Equation (7.9)]{Casas-Raymond2006} should be used in the last step.

First order optimality conditions read as
\begin{equation}\begin{array}{rcl}
\matriz{A}\bm{u}+\matriz{S}^T\matriz{L}\max\{0,\bm{\mu}^*+\gamma(\matriz{S}\bm{u}-\bm{b})\}+\bm\lambda &= &\bm{f},\\
\bm{\lambda}-\max(0,\bm{\lambda}+c(\bm{u}-\bm{\beta}))&=&0.
\end{array}\label{xxxE54}
\end{equation}
Using the definitions \eqref{E4.14} and \eqref{E5.5} for the active sets of indexes $\indice{A}(\bm{u},\bm{\lambda})$ and $\indice{A}_\omega(\bm{u},\gamma,h)$ and the matrix $\matriz{H}$ related to $\indice{A}_\omega$, we have that one step of Newton's method can be written as
\begin{equation}\begin{array}{rcl}
 \bm{u}^+_\indice{F} &=& \displaystyle\arg\min_{\bm{u}_\indice{F}} \frac12 \bm{u}_\indice{F}^T
(\matriz{I}_{\indice{F},\indice{B}}(\matriz{A} + \gamma\matriz{S}^T\matriz{H} \matriz{S})\matriz{I}_{\indice{B},\indice{F}})
\bm{u}_\indice{F} \\ \\ && \phantom{xxxxxx} -
\big( \matriz{I}_{\indice{F},\indice{B}}
(\,
  \bm{f}
- \matriz{A}\matriz{I}_{\indice{B},\indice{A}} \bm{\beta}_\indice{A}
- \matriz{S}^T\matriz{H} ( \bm{\mu}^* +\gamma (\matriz{S} \matriz{I}_{\indice{B},\indice{A}} \bm{\beta}_\indice{A} -\bm{b}) )
\,)
\big)^T
\bm{u}_\indice{F} \\
\bm{u}^+_\indice{A}&=& \bm{\beta}_\indice{A}
\\
\bm{\lambda}^+& =&  \bm{f} - \matriz{A} \bm{u}^+ - \matriz{S}^T\matriz{H} ( \bm{\mu}^* +\gamma (\matriz{S}\bm{u}^+ -
\bm{b}))
\end{array}\label{xxxE56}\end{equation}
or, alternatively, as
\begin{equation}
\begin{array}{rcl}
\left[\begin{array}{ccc}
\matriz{M}+\nu\matriz{B}+\gamma\matriz{H} & -\matriz{K}_{:,\indice{I}}  & \matriz{I}_{:,\indice{A}} \\
-\matriz{K}_{\indice{I},:} & \matriz{O}_{\indice{I},\indice{I}} & \matriz{O}_{\indice{I},\indice{A}}\\
\matriz{I}_{\indice{B},:} & \matriz{O}_{\indice{B},\indice{I}} & \matriz{O}_{\indice{B},\indice{A}}\\

\end{array}\right]
\left[\begin{array}{c}
\bm{y}^+ \\
\bm{\varphi}_\indice{I}^+ \\
\bm{\lambda}^+\end{array}
\right]
&=&
\left[\begin{array}{c}
\matriz{M}\bm{y}_\Omega+\matriz{H}(\gamma\bm{b}-\bm{\mu}^*)\\
\bm{0}\\
\matriz{I}_{\indice{A},\indice{B}} \bm{\beta}
\end{array}
\right]\\
\bm{u}^+ &=& \bm{y}^+_{\indice{B}}\\
\bm{\lambda}^+_{\indice{F}} &=&\bm{0}.
\end{array}
\label{xxxE55}
\end{equation}

An adaptation of algorithms \ref{A450} and \ref{A600} to solve $(Q^{C,\refre{\gamma}}_h)$ is straightforward, and so is an adaptation of Algorithm \ref{A1100} to use a continuation strategy together with a nested mesh strategy.
\begin{example}We repeat Example \ref{Ex5.7} adding the control constraint $u\leq 0.16$. We obtain the results summarized in Table \ref{xxxT10}. The total computation time was 2894 seconds. Although this may seem a lot of time, we have to remember that we are solving an optimal control problem with pointwise constraints in both the control and the state using a mesh with almost 2.7E+06 nodes and 5.4E+06 elements to test the algorithm.
\begin{table}[h!]
\tbl{Convergence history for the solution of $(P^{CS})$.}
{
\begin{tabular}{cr|c|c|c|c|l}\toprule
&$j$ & $\gamma_n$ & $r_d$ & Newton & pcg & $J_h(u_h^\gamma)$ \\ \midrule
& 0 &      1E+00 & 9.42E-04 & 2 &  10 & 0.3527536758\\
& 1 &      1E+01 & 1.12E-03 & 2 &  10 & 0.3530056222\\
& 2 &     1E+02 & 9.03E-04 & 3 &  15 & 0.3535235000\\
& 3 &     1E+03 & 2.26E-04 & 4 &  33 & 0.3546079241\\
& 4 &    1E+04 & 2.38E-05 & 4 &  40 & 0.3551408902\\
& 5 &    1E+05 & 2.36E-06 & 6 &  59 & 0.3552411683\\
& 6 &   1E+06 & 2.34E-07 & 6 &  65 & 0.3552759281\\
&7 &  1E+07 & 2.33E-08 & 8 &  94 & 0.3552874384\\
& 8 &   1E+08 & 2.32E-09 & 9 &  104 & 0.3552912421\\
\bottomrule
\end{tabular}
}
\label{xxxT10}
\end{table}

\end{example}
\section*{Funding}
The author was partially
supported by the Spanish Ministerio de Ciencia e Innovaci\'on under Project MTM2014-57531-P.


\end{document}